\documentclass[a4paper,11pt]{amsart}
\usepackage{etex}
\usepackage{amsmath,amsfonts,amssymb,amsthm}
\usepackage{pstricks}
\usepackage{subfig}
\usepackage{graphicx}
\usepackage{color}
\usepackage[all]{xy}
\usepackage{url}
\usepackage{alltt}
\usepackage{setspace}
\usepackage{booktabs}
\usepackage{longtable}
\usepackage{latexsym}
\usepackage[linesnumbered]{algorithm2e}
\usepackage{tikz}
\usetikzlibrary{cd}
\usetikzlibrary{patterns}
\usetikzlibrary{shapes}
\usetikzlibrary{decorations.fractals}
\usepackage[textsize=tiny]{todonotes}

\newtheorem{theorem}{Theorem}
\newtheorem{lemma}[theorem]{Lemma}

\newtheorem{proposition}[theorem]{Proposition}

\theoremstyle{definition}
\newtheorem{definition}[theorem]{Definition}
\newtheorem{example}[theorem]{Example}

\newtheorem{remark}[theorem]{Remark}

\definecolor{oliwkowy}{HTML}{627037}
\definecolor{cKlaus}{rgb}{0.0,0.55,0.63}
\definecolor{loesung}{rgb}{0.6,0.10,0.33}
\definecolor{cKlausOK}{rgb}{0.3,0.40,0.33}
\definecolor{intOrange}{rgb}{1.0,.310,.0}  
                   
\definecolor{cMakro}{rgb}{1,0,0}  
\definecolor{darkgreen}{RGB}{0, 128, 0}

\definecolor{cFrederik}{rgb}{0.3,0.40,0.53}

\definecolor{cJan}{rgb}{0.1,0.20,0.73}

\newcommand{\peeref}[1]{{#1}}

\newcommand{\idm}{\mathfrak{m}}

\newcommand{\PP}{\mathbb P}

\newcommand{\A}{\mathbb A}

\newcommand{\N}{\mathbb N}
\newcommand{\R}{\mathbb R}
\newcommand{\T}{\mathbb T}
\newcommand{\V}{\mathbb V}
\newcommand{\Z}{\mathbb Z}
\newcommand{\C}{\mathbb C}

\renewcommand{\P}{\mathbb P}

\newcommand{\CA}{{\mathcal A}}

\newcommand{\CE}{{\mathcal E}}
\newcommand{\CF}{{\mathcal F}}

\newcommand{\CH}{{\mathcal H}}

\newcommand{\CO}{{\mathcal O}}

\newcommand{\CT}{{\mathcal T}}

\DeclareMathOperator{\spec}{Spec}

\DeclareMathOperator{\supp}{supp}
\DeclareMathOperator{\id}{id}

\DeclareMathOperator{\conv}{conv}

\DeclareMathOperator{\trace}{tr}

\DeclareMathOperator{\rank}{rk}

\DeclareMathOperator{\orb}{orb}
\DeclareMathOperator{\End}{End}

\newcommand{\til}[1]{\widetilde{#1}}
\renewcommand{\iff}{\Longleftrightarrow}

\newcommand{\then}{\Rightarrow}

\newcommand{\gHom}{\operatorname{Hom}}

\newcommand{\gEnd}{\operatorname{End}}
\newcommand{\lEnd}{End}

\newcommand{\spann}{\operatorname{span}}

\newcommand{\kst}{\,|\;}
\newcommand{\kSt}{\,\big|\;}

\newcommand{\ku}{\underline}
\newcommand{\ks}{\scriptstyle}

\newcommand{\surj}{\rightarrow\hspace{-0.8em}\rightarrow}

\newcommand{\kss}{\scriptscriptstyle}
\newcommand{\kbb}{{\kss \bullet}}
\newcommand{\ko}{\overline}

\newcommand{\toric}{\T\V}

\newcommand{\HR}{\CH}

\newcommand{\coHikks}{\Phi} 
\newcommand{\cohikks}{\phi} 
\newcommand{\cohikksT}{\phi} 
\newcommand{\tcohikks}{\til{\phi}} 
\newcommand{\PhiAlg}{\CA} 
\newcommand{\cHAlg}{\PhiAlg}
 
\newcommand{\kbox}{\mbox}
\begin{document}

\title[Toric co-Higgs sheaves]
{Toric co-Higgs sheaves}
\author[K.~Altmann]{Klaus Altmann
}
\address{Institut f\"ur Mathematik,
FU Berlin,
Arnimalle 3,
D-14195 Berlin,
Germany}
\email{altmann@math.fu-berlin.de}
\author[F.~Witt]{Frederik Witt}
\address{
Fachbereich Mathematik,
Universit\"at Stuttgart,
Pfaffenwaldring 57,
D-70569 Stuttgart, 
Germany}
\email{witt@mathematik.uni-stuttgart.de}
\thanks{MSC 2010: 14M25; 
52B20, 
14J60; 
Key words: toric variety, co-Higgs sheaves, filtrations, polytopes}

\begin{abstract}
We characterise and investigate co-Higgs sheaves and associated algebraic and
combinatorial invariants on toric varieties. In particular, we compute
explicit examples.
\end{abstract}

\maketitle

\section{Introduction}
\label{intro}

\subsection{Co-Higgs sheaves on toric varieties}
\label{toricCoHiggs}
\peeref{Let $X$ be a normal variety over $\C$ and $\CE$ 
a reflexive sheaf of $\CO_X$-modules. 
Consider an $\CO_X$-linear map 
$\Phi:\CE\to\CE\otimes_{\CO_X}\CT_X$, where $\CT_X$ 
denotes the tangent sheaf of $X$. We define the homomorphism 
$\Phi\wedge\Phi:\CE\to\CE\otimes_{\CO_X}\Lambda^2\CT_X$ 
as the composition
\begin{equation}
\label{eq:def.wedge}
\xymatrix{
\CE\ar[r]^{\Phi\phantom{aaaa}}&\CE\otimes\CT_X\ar[r]^{\Phi\otimes\id\phantom{aaaa}}&\CE\otimes\CT_X\otimes\CT_X\ar[r]^{\quad\id\otimes\wedge^2\phantom{a}}&\CE\otimes\Lambda^2\CT_X.
}
\end{equation}

\begin{definition}
\label{def:Higgs}
If $\Phi\wedge\Phi=0$ holds, $(\CE,\Phi)$ is called a {\em \kbox{co-Higgs} sheaf};
$\Phi:\CE\to\CE\otimes\CT_X$ is refered to as the {\em co-Higgs field}.
\end{definition}

In contrast, a {\em Higgs sheaf} is given by an $\CO_X$-linear map 
$\Psi:\CE\to\CE\otimes_{\CO_X}\Omega_X^1$ with $\Psi\wedge\Psi=0$, where
$\Omega_X^1$ denotes the cotangent sheaf of $X$.
The latter notion goes back to~\cite{hi87} and~\cite{si94I}
while co-Higgs sheaves originated in~\cite{hi11} and~\cite{ra11} 
in the context of Hitchin's generalised geometries. Indeed, if $X$ is smooth, 
co-Higgs sheaves with $\CE=\CT_X$ have a natural 
interpretation in terms of generalised complex structures on $X$~\cite{gu11,hi11}.

\medskip

In this article we study co-Higgs sheaves $(\CE,\Phi)$ 
over a complete normal toric variety $X$ with character lattice $M$ and
$\CE$ a toric sheaf. In particular, we can 
appeal to Klyachko's description of such sheaves as does 
the very recent article~\cite{bdmr20}, 
cf.~Section~\ref{toric-co-Higgs}.
However, in contrast to \cite{bdmr20}, we do not only focus on 
invariant, i.e., $M$-homogeneous co-Higgs fields, but on 
completely general ones. This leads to combinatorial invariants such as 
the {\em Higgs polytope} and the {\em Higgs range} which reflect 
the position of possible multidegrees of co-Higgs fields.

\subsection{Plan of the paper}
\label{planPaper}
We briefly summarise our main results and the content of the paper. 

\medskip

Section~\ref{toric-co-Higgs} reviews Klyachko's classification 
of toric sheaves and discusses some examples. 

\medskip

Theorem~\ref{th-describeHiggs} in Section~\ref{HiggsAlgebra} 
gives rise to a combinatorial description of  
co-Higgs sheaves based on Klyachko's formalism. 
We then briefly neglect the integrability condition $\Phi\wedge\Phi=0$ and 
focus on general $\CO_X$-linear maps $\Phi:\CE\to\CE\otimes\CT_X$ 
which we call {\em pre}-co-Higgs fields.
The reason for this is that pre-co-Higgs fields behave well under
sums, i.e., under decomposition into homogeneous components.
Afterwards, we use the integrability condition $\Phi\wedge\Phi=0$ to define a 
family of endomorphism algebras paramatrised by the torus 
(Proposition~\ref{prop:AlgT}). 

\medskip

Section~\ref{sec:CombInv} introduces two combinatorial invariants. 
First, a toric co-Higgs field gives rise to the {\em Higgs polytope} 
in the character lattice by taking the convex hull of its homogeneous 
degrees. Second, the totality of degrees of all possible toric pre-co-Higgs 
fields defines a further polytope, the {\em Higgs range}. 

\medskip

The computation of the Higgs range for co-Higgs sheaves $(\CT_X,\Phi)$ 
on a smooth toric surface $X$ will be the endeavour 
for the rest of this paper. In Section~\ref{revP2} we explain 
the computation for the projective plane $\PP^2$, 
cf.\ Theorem~\ref{th-HiggsRangePb}. 
Further, we sketch the case of 
Hirzebruch and Fano surfaces in Section~\ref{sec:MInSur}. 
Finally, we exhibit explicit Higgs polytopes on del Pezzos of 
degree $6$ and $7$ in Subsection~(\ref{subsec:HiggsPolyDP}). 
In these examples every subpolytope of the Higgs range 
can be realized as the Higgs polytope of some toric co-Higgs field.

\subsection{Convention}
\label{noCo}
As we will work exclusively with co-Higgs sheaves 
(with Remark~\ref{rem:NilHiggs} as sole exception)
\begin{center}
\fbox{we drop the qualifier ``co-'' in the sequel}
\end{center}
to simplify language. Hence we speak about toric Higgs sheaves 
when we actually mean toric co-Higgs sheaves etc.
Hopefully no confusion will arise.

\subsection{Acknowledgements}
\label{ack}
We thank Jan Christophersen for initial collaboration and 
the referee for detailed comments on the manuscript.
Furthermore, Christian Drosten who 
accompanied this project in the last weeks.}

\section{Klyachko's formalism}
\label{toric-co-Higgs}
Klyachko's description of toric vector bundles and, more generally,
of toric reflexive sheaves appeared in \cite{klyachko},
see also \cite{klyachkoICM} for his ICM 2002 talk on this subject.
A short summary can be found in~\cite{payne}.
Further, more recent approaches can be found in \cite{parliaments}
and \cite{tropical}.

\subsection{Klyachko's description of toric sheaves}
\label{Klyachko}
Consider a toric variety $X=\toric(\Sigma)$ given by a fan $\Sigma$ in $N_\R=N\otimes_\Z\R$, where $N$ is a lattice of rank $q$. As usual, its dual $M=\gHom_\Z(N,\Z)$ denotes the {\em character lattice}. Then $X$ contains the torus $T=N\otimes_\Z\C^*$ and we may pick the neutral element $1\in T\subseteq X$. Each $\CO_X$-module $\CE$ gives rise to a $\C$-vector space
$$
E:=\CE(1):=\CE_1/\idm_{X,1}\CE_1,
$$
where $\CE_1$ denotes the stalk of $\CE$ at $1\in X$ and $\idm_{X,1}$ the maximal ideal of $1$. If $\CE$ is a $T$-equivariant, i.e., $T$-linearized, torsion free sheaf on $X$, the global sections of $\CE$ are an $M$-graded subset of $E\otimes_\C\C[M]$. If, in addition, $\CE$ is reflexive, then $\CE$ is already determined 
by its restriction to open subsets whose complements are of codimension equal or greater than two. We briefly refer to $\CE$ as a {\em toric sheaf}.

\medskip

Via Klyachko's description~\cite{klyachko}, a toric sheaf $\CE$ corresponds to a set of decreasing $\Z$-filtrations
$$
E_\rho^\kbb \;=\; [\ldots \supseteq E_\rho^{\ell-1}\supseteq E_\rho^{\ell}\supseteq 
E_\rho^{\ell+1}\supseteq\ldots]\hspace{1em}(\ell\in\Z)
$$
of the vector space $E$ which are \kbox{parametrized by the rays} or one-dimensional cones $\rho\in\Sigma(1)$. By abuse of notation we use $\rho$ for both the ray and its primitive generator. The filtrations encode the sections of $\CE$ on the $T$-invariant open subsets $U_\rho=\toric(\rho)\subseteq X$ defined by $\rho$. Namely, for $r\in M$, we have
$$
e\otimes\chi^r\in\Gamma(U_\rho,\CE)\quad\iff\quad 
e\in E_\rho^{-\langle r,\rho\rangle}.
$$
Since $\bigcup_{\rho\in\Sigma(1)}U_\rho$ is an open  
\peeref{subset of $X$ just missing out a set}
of codimension at least two,
\begin{equation}
\label{eq:poss.deg}
e\otimes\chi^r\in\Gamma(X,\CE)\quad\iff\quad e\in\bigcap_{\rho\in\Sigma(1)}E_\rho^{-\langle r,\rho\rangle}.
\end{equation}

\begin{remark}
\label{rem:klya.comp}
The reflexive sheaf $\CE$ \peeref{is {\em locally free}} if it is 
subject to Klyachko's compatibility condition~\cite{klyachko}: 
For each cone $\sigma\in\Sigma$ there exists a decomposition 
$$
E =\bigoplus_{[u]\in M/(M\cap\sigma^\perp)}
\peeref{E_{[u]}^{(\sigma)}}
$$ 
such that $E^l_\rho = \sum_{\langle u,\rho\rangle\ge l} \peeref{E_{[u]}^{(\sigma)}}$ 
for each ray $\rho$ contained in $\sigma$. 
\peeref{Thus, for each $\sigma$,
the dimensions of the spaces $E_{[u]}^{(\sigma)}$ add up to
$\dim E=\rank\CE$. This can be reformulated as follows. The multisets
$\ku{u}(\sigma)\subset M/(M\cap\sigma^\perp)$ 
of cardinality $\rank\CE$ such that each $[u]$ appears exactly with multiplicity
$\dim E_{[u]}^{(\sigma)}$ are in bijection with isomorphism classes of locally free 
toric sheaves~\cite[Corollary 2.3]{payne}. 
Furthermore, the elementary symmetric functions 
in $\ku{u}(\sigma)$ glue together 
to piecewise polynomial functions on $N$ 
which encode the equivariant Chern classes of 
$\CE$~\cite[Section 3.1]{payne}.
We will see that they control the size of the Higgs range, see 
Definition~\ref{def-HiggsRange} and Remark~\ref{rem:chernRoot}.
}
\end{remark}

\begin{example}
\label{ex-Klyachko}
\peeref{Consider a smooth toric variety $X$ together with the following \peeref{locally free toric sheaves} (cf.\ \cite[Example 2.3]{klyachko}):}
\begin{enumerate}
\item[(i)] Let $D_\rho=\ko{\orb(\rho)}$ be the closure of the orbit defined by
the ray $\rho$. For $D=\sum_\rho \lambda_\rho\,D_\rho$, the invertible sheaf
$\CO(D)$ is encoded by
$$
E_\rho^\ell:=\left\{\begin{array}{cl}
\C & \mbox{if } \ell \leq\lambda_\rho
\\
0 & \mbox{if } \ell \geq \lambda_\rho+1
\end{array}\right\}
\subseteq \C=:E.
$$
\item[(ii)] The cotangent sheaf $\Omega^1_X$ corresponds to the filtration
$$
E_\rho^\ell:=\left\{\begin{array}{cl}
M_\C & \mbox{if } \ell \leq -1
\\
\rho^\bot & \mbox{if } \ell =0
\\
0 & \mbox{if } \ell \geq 1
\end{array}\right\}\subseteq M_\C=:E.
$$
\item[(iii)] On the other hand, the tangent sheaf $\CT_X$ corresponds to the filtration
$$
T^{\ell}_\rho:=\left\{\begin{array}{cl}
N_\C & \mbox{if } \ell \leq 0
\\
\spann(\rho) & \mbox{if } \ell =1
\\
0 & \mbox{if } \ell \geq 2.
\end{array}\right\}\subseteq N_\C=:E.
$$ 
\end{enumerate}
For instance, for $\PP^1$ we recover the first example since 
$\CT_{\PP^1}=\CO_{\PP^1}(D_{[1:0]}+D_{[0:1]})$. 
In fact, Examples (ii) and (iii) are connected via the general formula 
relating the filtrations of an equivariant reflexive sheaf with its dual sheaf. 
\end{example}

We can use the description of Example~\ref{ex-Klyachko} to calculate the
global sections of various toric sheaves using~\eqref{eq:poss.deg}.

\begin{example}
\label{ex-KlyachkoGlob}
For further use we consider the twisted tangent sheaf $\CT(d)$ over $\PP^2$. 
The Euler sequence immediately yields that $\Gamma(\PP^2,\CT_{\PP^2}(d))$ 
is a $h(d)=(d^2+6d+8)$-dimensional complex vector space if $d\geq-1$ and trivial otherwise. To derive this from a toric point of view, we let 
\begin{equation}
\label{eq:fan.pp2}
a=(1,0),\quad b=(0,1)\quad\mbox{and}\quad c=(-1,-1)
\end{equation}
denote the rays of the fan of $\PP^2$. 
The filtration is given by $F^\ell_\rho=\sum_{i+j=\ell}E^i_\rho\otimes T^j_\rho$ where $T^{\ell}_\rho$ and $E_\rho^\ell$ are the filtrations of the tangent sheaf $\CT_{\PP^2}$ and of the 
invertible sheaf $\CO(dD_a)$ respectively. In particular, 
\[
F_a^\ell=\left\{\begin{array}{cl}
\C^2& \mbox{if } \ell \leq d
\\
\spann(a) & \mbox{if } \ell =d+1.
\\
0 & \mbox{if } \ell\geq d+2.
\end{array}\right.
\qquad
F_{b,c}^\ell=\left\{\begin{array}{cl}
\C^2& \mbox{if } \ell \leq 0
\\
\spann(b),\spann(c) & \mbox{if } \ell=1.
\\
0 & \mbox{if } \ell\geq2.
\end{array}\right.
\]
Then $0\not=f\in\bigcap_{\rho\in\Sigma(1)}F^{-\langle r,\rho\rangle}$ 
implies the inequalities $r_1\geq-d-1$, $\,r_2\geq-1$, and $\,r_1+r_2\leq1$. 
Now the vertices of this polytope cannot give rise to nontrivial 
sections as $\spann(\rho)\cap\spann(\rho')=0$ for rays $\rho\not=\rho'$. 
The $3\cdot(d+2)$ lattice points in the facets span a one-dimensional space. On the other hand, the $(d+1)(d+2)/2$ interior lattice points span a two-dimensional space each. Hence $h(d)=3(d+2)+(d+2)(d+1)=d^2+6d+8$ in accordance with the Euler formula.
\end{example}

\subsection{Klyachko's description of morphisms between toric sheaves}
Next assume that $\CE$ and $\CF$ are two toric sheaves over $X=\toric(\Sigma)$ given by filtrations $E_\rho^\ell\subseteq E$ and $F_\rho^\ell\subseteq F$, $\rho\in\Sigma(1)$, respectively. The space of homomorphisms $\CE\to\CF$ is 
therefore graded over $M$,
$$
\gHom(\CE,\CF)=\bigoplus_{r\in M}\gHom_T(\CE,\CF[r])
\subseteq\gHom(E,F)\otimes\C[M].
$$
Here, $\gHom_T(\cdot\,,\cdot)$ denotes the $T$-equivariant morphisms,
and $\CF[r]$ is the toric sheaf $\CF$ with the new $T$-action obtained by 
\kbox{twisting with the character $\chi^r$}. In particular, an equivariant $\Phi\in\gHom_T(\CE,\CF)$ corresponds to a linear map $\phi\in\gHom(E,F)$ which satisfies $\phi(E^\ell_\rho)\subset F^\ell_\rho$ for all $\rho\in\Sigma(1)$ and $\ell\in\Z$. Since the filtration of $\CF[r]$ is given by
$$
\peeref{F[r]^\ell_\rho:= }
F_\rho^{\ell-\langle r,\rho\rangle} \subseteq F,\quad\rho\in\Sigma(1),
$$
an equivariant $\Phi\in\gHom_T(\CE,\CF[r])$ is given by $\phi\otimes\chi^r$ with
\begin{equation}
\label{eq:degree_r}
\phi(E_\rho^\ell)\subseteq \peeref{F[r]^\ell_\rho}\quad\mbox{ for all }\rho\in\Sigma(1)\mbox{ and }\ell\in\Z.
\end{equation}
A general homomorphism $\Phi\in\gHom(\CE,\CF)$ is the sum $\Phi=\sum_{r\in M}\Phi^r$ of {\em homogeneous homomorphisms of degree $r$} with $\Phi^r=\phi^{r}\otimes\chi^r$, where $\phi^{r}\in\gHom(E,F)$ is the {\em associated $\C$-linear map}.

\begin{example}\label{ex:mor.tor}
Let us compute a basis for $\gHom(\CO(1),\CT_{\PP^2})\cong\Gamma(\PP^2,\CT_{\PP^2}(-1))$ using the notation from Example~\ref{ex-KlyachkoGlob}. 
The invertible sheaf $\CO(1)=\CO(D_a)$ is given by
\[
E_a^\ell=\left\{\begin{array}{cl}
\C& \mbox{if } \ell \leq1
\\
0 & \mbox{if } \ell\geq2.
\end{array}\right.
\qquad
E_{\rho=b,c}^\ell=\left\{\begin{array}{cl}
\C& \mbox{if } \ell \leq 0
\\
0 & \mbox{if } \ell\geq1.
\end{array}\right.
\]
while we find
\[
T_{\rho=a,b,c}^\ell=\left\{\begin{array}{cl}
\C^2& \mbox{if } \ell \leq0,
\\
\spann(\rho) & \mbox{if } \ell =1.
\\
0 & \mbox{if } \ell\geq2
\end{array}\right.
\]
for the tangent sheaf. If
$0\not=\Phi^r=\phi^{r}\otimes\chi^r\in\gHom(\CO_{\PP^2}(1),\CT_{\PP^2})$ is
nontrivial, then $\phi^{r}(E^\ell_\rho)$ is nontrivial if $\ell-\langle
r,\rho\rangle\leq1$ for $\ell\leq1$ if $\rho=a$ and $\ell\leq0$ if $\rho=b$,
$c$. This is equivalent to the inequalities $r_1\geq0$, $r_1\geq-1$ and
$r_1+r_2\leq1$. Excluding the vertices of the resulting polytope we find the
same result as in Example~\ref{ex-KlyachkoGlob}.
\end{example}

\section{Toric Higgs sheaves}
\label{HiggsAlgebra}

\subsection{Toric pre-Higgs fields}
\label{Klyachko-Higgs}
We now focus on the case $\CF=\CE\otimes_{\CO_X}\CT_X$. 
We recall our convention from the introduction that we drop the qualifier ``co-'' although we tensor with $\CT_X$, not with $\Omega^1_X$. 

\begin{definition}
\label{def-pre-Higgs}
A {\em pre-Higgs field} $\Phi$ on $\CE$ is a morphism in 
$$
\gHom(\CE,\CE\otimes_{\CO_X}\CT_X)=\bigoplus_{r\in M}
\gHom_T(\CE,\CE\otimes_{\CO_X}\CT_X)[r].
$$ 
It is thus the direct sum $\Phi=\sum\Phi^r$ of 
$M$-homogeneous maps $\Phi^r:\CE\to\CE\otimes_{\CO_X}\CT_X$ of 
degree $r\in M$. The $\Phi^r$ 
are called {\em homogeneous pre-Higgs fields}. 
Writing $\Phi^r=\phi^{r}\otimes\chi^r$ 
we obtain the {\em associated $\C$-linear map} $\phi^{r}:E\to E\otimes N_\C$. The pair $(\CE,\Phi)$ is called a {\em toric pre-Higgs sheaf}.
\end{definition}

We need to analyse the filtrations $F_\rho^{\kbb}$ of 
$\CF=\CE\otimes\CT_X$ next. These fit into the following exact sequence:

\begin{lemma}
\label{lem:exact.sequence}
For $\rho\in\Sigma(1)$ the sequence 
$$
\xymatrix{
0\ar[r]&F_\rho^{\ell}\ar[r]& E_\rho^{\ell-1}\otimes N_\C \ar[r]^{\hspace{-35pt}\pi_\rho}&\big(E_\rho^{\ell-1}/E_\rho^{\ell}\big)\otimes\big(N_\C/\spann(\rho)\big)\ar[r]& 0,
}
$$
where $\pi_\rho$ denotes the natural projection, is exact. 
\begin{proof}
If $T^{\ell}_\rho$ denotes the filtration of the tangent sheaf 
(cf.\ Example~\ref{ex-Klyachko} (iii)), 
then $F_\rho^{\ell}=\sum_{i+j=\ell}E^i_\rho\otimes T^j_\rho$. Since $E^i_\rho\otimes T^j_\rho\subset E^{\ell}_\rho\otimes N_\C$ whenever $i=\ell-j$, $j\leq 0$, we have $F_\rho^{\ell}=E_\rho^{\ell}\otimes N_\C + E_\rho^{\ell-1}\otimes\spann(\rho)$. In particular, we have a natural injection $F^{\ell}_\rho\to E^{\ell-1}_\rho\otimes N_\C$. Clearly, $F^{\ell}_\rho\subseteq\ker\pi_\rho$; equality follows on dimensional grounds.
\end{proof}
\end{lemma}

Given a map $\phi:E\to E\otimes N_\C$ it will be convenient to consider the 
contraction of $\phi$ by $s\in M$, namely 
$$
\kbox{$\cohikks_s
:= \langle s,\phi\rangle\peeref{:=(\id_E\otimes s)\circ\phi}\in\End(E)$.}
$$
\peeref{Here, $s\in M$ is understood as a $\C$-linear map
$N_\C\to\C$. Further, we sometimes
regard $\phi$ as a $\Z$-linear map $M\to\gEnd(E)$, $s\mapsto\cohikks_s$.}

\begin{theorem}
\label{th-describeHiggs}
A $\C$-linear map $\phi:E \to E\otimes N_\C$ induces a homogeneous pre-Higgs field of degree $r$ if and only if the associated contractions satisfy 
$$
\cohikks_s(E_\rho^\ell)\subseteq
\left\{\begin{array}{ll}E_\rho^{\ell-\langle r,\rho\rangle}&
\mbox{if }s\in\rho^\perp\\
E_\rho^{\ell-1-\langle r,\rho\rangle}&
\mbox{if }s\not\in\rho^\perp\end{array}\right\}
\subseteq E_\rho^{\ell-1-\langle r,\rho\rangle}
$$
for all $s\in M$, $\rho\in\Sigma(1)$ and $\ell\in\Z$.
\end{theorem}

\begin{proof}
If $\phi$ is $\C$-linear, then 
$\phi(E^l_\rho)\subset F^{l-\langle r,\rho\rangle}=
E_\rho^{\ell}\otimes N_\C + E_\rho^{\ell-1}\otimes\spann(\rho)$. 
Hence this is a necessary condition. 
Conversely, we already know that 
$\phi(E^{\ell}_\rho)\subseteq E^{l-1}_\rho\otimes N$ for 
$\cohikks_s(E^{\ell}_\rho)\subseteq E^{\ell-1-\langle r,\rho\rangle}_\rho$. 
To show that the image actually lies in 
$F^{\ell-\langle r,\rho\rangle}_\rho\otimes N$ we need to prove that 
$\pi_\rho\circ\phi=0$. Considering this as a map 
$(N_\C/\spann(\rho))^*=\rho^\perp\to E^{\ell-1-\langle r,\rho\rangle}/
E^{\ell-\langle r,\rho\rangle}$ this must vanish. 
Since $\cohikks_s(E^{\ell}_\rho)=\langle s,\phi(E^{\ell}_\rho)\rangle$ 
this holds by assumption.
\end{proof}

\begin{remark}\label{rem:NilHiggs}
\peeref{In the same vein we can analyse $\CO_X$-linear maps 
$\Psi:\CE\to\CE\otimes_{\CO_X}\Omega^1_X$ with 
$\CT_X$ replaced by $\Omega^1_X$. 
Here, a homogeneous morphism $\Psi:\CE\to\CE\otimes_{\CO_X}\Omega^1_X$ 
of degree $r$ corresponds to a $\C$-linear map $\psi:E \to E\otimes_\C M_\C$
satisfying $\psi_b(E_a^{\ell})\subseteq E_a^{\ell-\langle r,a\rangle+\delta_{ab}}$
for all $a,\,b\in\Sigma(1)$ and $\ell\in\Z$, where 
$\delta_{ab}=1$ if and only if $a=b$. 
Therefore, the difference between using $\Omega^1$ (as one does for 
usual Higgs fields) and $\CT_X$ is just the sign in $\pm\delta_{ab}$. 

\medskip

However, taking $\Omega^1$ instead of $\CT_X$ implies that
the associated $\C$-linear map $\psi:E\to E\otimes M_\C$ 
gives rise to \kbox{nilpotent endomorphisms $\psi_a$}, $a\in\Sigma(1)$, if 
$r=0$ or $r\in M\setminus |\Sigma|^\vee$ (which is always the case if $X$ is 
complete). Indeed, for $r\in M\setminus |\Sigma|^\vee$ there is a $b\in\Sigma(1)$ 
such that $\langle r,b\rangle<0$, whence 
$\psi_a(E_b^\ell)\subseteq E_{b}^{\ell-\langle r,b\rangle}\subseteq E_{b}^{\ell+1}$. 
It follows, for instance, by a direct computation that 
the tangent sheaf $\CE:=\CT_{\PP^2}$ does not 
admit any nontrivial toric pre-Higgs field 
$\Psi:\CT_{\P^2}\to\CT_{\P^2}\otimes_{\CO_{\P^2}}\Omega^1_{\P^2}$ in 
striking contrast to the case of pre-co-Higgs fields 
$\Phi:\CT_{\P^2}\to\CT_{\P^2}\otimes_{\CO_{\P^2}}\CT_{\P^2}$.}
\end{remark}
\subsection{Toric Higgs fields}
\label{subsec:TorHiggs}
We return to our Convention~(\ref{noCo}) and come to the central definition 
of this paper. 
It introduces the equivariant versions of Definition~\ref{def:Higgs}.

\begin{definition}
\label{def-higgs}
\hfill
\begin{enumerate}
\item[(i)] A {\em toric Higgs sheaf} $(\CE,\Phi)$ consists of a toric sheaf $\CE$ over the toric variety 
$X=\toric(\Sigma)$, and an arbitrary, not necessarily homogeneous pre-Higgs field $\Phi:\CE\to\CE\otimes\CT_X$ which satisfies $\Phi\wedge\Phi=0$, cf.~\eqref{eq:def.wedge}. We refer to $\Phi$ as the {\em Higgs field} of the underlying {\em toric sheaf} $\CE$.
\item[(ii)] A {\em homogeneous Higgs sheaf} $(\CE,\Phi)$ is a toric Higgs sheaf with homogeneous Higgs field $\Phi$ of given degree $r\in M$. It corresponds to a $\C$-linear map $\phi:E\to E\otimes M_C$. 
\end{enumerate}
\peeref{By a slight abuse of notation} we speak of {\em Higgs bundles} instead of Higgs sheaves if $\CE$ is actually locally free.
\end{definition}

\begin{remark}
\label{rem:com.inher}
The condition $\Phi\wedge\Phi=0$ is not inherited by the homogeneous
components of $\Phi$, so that $\Phi=\sum\Phi^r$ is merely a
decomposition into homogeneous pre-Higgs fields. On the other hand, the
direct sum of homogeneous Higgs fields $\Phi=\sum\Phi^r$ is not
necessarily a Higgs field \peeref{either} as it might not satisfy
$\Phi\wedge\Phi=0$.
\end{remark}
\subsection{The Higgs algebra}
\label{subsec:HiggsAlg}
A pre-Higgs field can be considered as an element of 
$\gEnd(E)\otimes N_\C\otimes \C[M]$ via $\Phi=\sum_{r\in M}\Phi^r=\sum_{r\in M}\phi^{r}\otimes\chi^r$. In particular, we can contract $\Phi$ with $s\in M$ and $t\in T$ to obtain
$$
\coHikks_s:=\langle s,\Phi\rangle=
\sum_{r\in M}\cohikks^{r}_s\otimes\chi^r\in\gEnd(E)\otimes\C[M]\quad\mbox{and}\quad\coHikks_s(t)\in\gEnd(E).
$$
The condition $\Phi\wedge\Phi=0$ translates into

\begin{proposition}
\label{prop:AlgT}
For any $s$, $s'\in M$ we have $[\coHikks_s,\coHikks_{s'}]=0$ in 
$\gEnd(E)\otimes\C[M]$. 
In particular, every Higgs field defines a family
$$
\textstyle
\CA(t)=\C\Big[\coHikks_s(t)=
\sum_{r\in M}\chi^r(t)\,\cohikks^{r}_s\kSt s\in M\Big]\subset \gEnd(E),\quad t\in T
$$
of (commutative) finitely generated subalgebras with $\id_E$ as unit. 
\end{proposition}

\begin{proof}
This is actually a non-toric property in the following sense. 
Consider a general Higgs field $\Phi$ as in Definition~\ref{def:Higgs}. 
Locally, we can fix a base of vector fields 
$\{\nu_i\kst i\in I\}\subset\CT_X$ 
over $U$ and write 
$\Phi|_U=\sum_{i\in I}\coHikks_i\otimes\nu_i$ with 
$\coHikks_i\in\Gamma(U,\lEnd(\CE))$. 
It follows that $[\coHikks_i,\coHikks_j]=0$ for all $i$, $j\in I$, 
for $0=\Phi\wedge\Phi=\sum_{i<j}[\coHikks_i,\coHikks_j]\,
\nu_i\wedge\nu_j$. 
In fact, for every local section $\omega\in\Omega^1_X$ we can evaluate so that
$\nu_i(\omega)\in\CO_X$. We end up with a set of 
\kbox{commuting endomorphisms} 
$\coHikks_\omega=\sum_{i\in I}\nu_i(\omega)\Phi_i\in\Gamma(U,\lEnd(\CE))$. 
In the toric case where $U=T$ and $\Gamma(U,\lEnd(\CE))=E\otimes\C[M]$, 
local toric sections $\nu_i$ of $\CT_X$ correspond to elements $n_i\in N$, 
and contracting with $s$ yields 
$\coHikks_s=\sum_{r\in M}\cohikks^{r}_s\otimes\chi^r$.
\end{proof}

\begin{definition}
We call the finitely generated, commutative
$\C[M]$-subalgebra 
$$
\textstyle
\CA=\CA(\Phi):=\C[M]\big[\coHikks_s=
\sum_{r\in M}\cohikks^{r}_s\otimes \chi^r\kSt s\in M\big]
$$
the \kbox{{\em Higgs algebra}} with $\id_E\otimes\chi^0$ as unit element.
\end{definition}

\begin{example}
\label{ex-cHAlgPOne}
\peeref{For $X=\PP^1$, $M=\Z$ and there are two primitive rays $\rho_0=1$ and $\rho_\infty=-1$.
A co-Higgs field $\coHikks\in\C[M]=\C[x,x^{-1}]$ therefore looks like
$\coHikks=c_0+c_1\,x + c_{-1}\,x^{-1}$, whence
$\coHikks_t=c_0+c_1\,t+c_{-1}\,t^{-1}\in\C$ for $t\in T=\C^*$. 
Their minimal polynomials
$m_\coHikks(z)\in\C[x,x^{-1}][z]$ and
$m_{\coHikks_t}(z)\in\C[z]$ equal $m_\coHikks(z)=z-\coHikks$ and
$m_{\coHikks_t}(z)=z-\coHikks_t$, respectively. This fits with
$$
\cHAlg(\coHikks)=\C[x,x^{-1}][\coHikks]=\C[x,x^{-1}]=
\C[x,x^{-1}][z]/m_\coHikks(z),
$$
and similarily for $\cHAlg_t$.}

\medskip

\peeref{Further examples will be discussed in Subsection~\eqref{subsec:HiggsPolyDP}.}
\end{example}

For each $t\in T$, we obtain a surjection
$\CA\surj\CA(t)$ within
$\gEnd(E)\otimes\C[M]\surj\gEnd(E)$.
We obtain the following commutative diagram where only the rightmost
column is non-commutative:
$$
\xymatrix@R5ex{
\coHikks_s \ar@{|->}[dd] &
\C[M] \ar@{->>}[d]_{t\in T} \ar@{^(->}[r]^-{a} &
\CA\; \ar@{^(->}[r]\ar@{^(->}[r]^-{\psi} \ar@{->>}[d]^-{}&
\gEnd(E)\otimes\C[M] \ar@{->>}[d]^-{}
\\
&
\C \ar@{^(->}[r] \ar@{=}[d]&
\CA\otimes_t\C\; \ar[r]^-{\psi_t} \ar@{->>}[d]^-{}&
\big(\gEnd(E)\otimes\C[M]\big)\otimes_t\C \ar@{=}[d]
\\
\coHikks_s(t) &
\C \ar@{^(->}[r] &
\CA(t)\; \ar@{^(->}[r] &
\gEnd(E)
}
$$
The injectivity of $\psi_t$ is equivalent to $\CA\otimes_t\C\to\CA(t)$ being an isomorphism which corresponds to the flatness of $a$.

\begin{remark}
\label{rem-commAlg}
The construction from the proof of Proposition~\ref{prop:AlgT} yields in the 
non-toric setting a sheaf of commutative
subalgebras of $\lEnd(\CE)$. The associated relative spectrum $\til{X}\to X$
relates to the spectral variety corresponding with the Higgs sheaf, 
cf.~\cite[Section 6]{si94II} and Subsection~\eqref{subsec:HiggsPolyDP} 
for an example in the toric setting. Note, however, that in contrast to the latter one,
our algebra involves the minimal polynomial instead of the characteristic one. 
Implicitely, we are using the isospectral decomposition of $E$ via characters 
providing the eigenvalues. 
Since there might be summands of dimension $>1$, this obstructs 
the construction of an honest fibration over $X$ (the spectral
variety). In the toric case it would be interesting to see how this sheaf 
relates to the algebra $\CA$ which is just the restriction to $T$. 
\end{remark}

\section{Combinatorial invariants}
\label{sec:CombInv}

\subsection{The Higgs polytope}
\label{sec:HiggsPoly}
For a given toric pre-Higgs sheaf $(\CE,\Phi)$ we can define a 
combinatorial invariant as follows.  
Let $\supp\Phi=\{r\in M\kst \Phi^r=\phi^{r}\otimes\chi^r\neq 0\}$ 
be the {\em support} of the pre-Higgs field $\Phi$.

\begin{definition}
\label{def:co.higgs.poly}
The convex lattice polytope in $M_\R$ defined by
$$
\nabla(\Phi):=\conv\supp\Phi,
$$
is called the {\em Higgs polytope} of $(\CE,\Phi)$. 
\end{definition}

This combinatorial invariant does heavily depend on the toric data. 
Whenever $\nabla'\subseteq\nabla$ is a subpolytope, e.g.,
$\nabla'=\{r\}$ for a single $r\in\nabla\cap M$, then we define the
restriction of a pre-Higgs field $\Phi$ to $\nabla'$ by
$$
\Phi|_{\nabla'}:=\sum_{r\in \nabla'} \Phi^r.
$$
Obviously, this defines again a pre-Higgs field.
In contrast, even for an honest Higgs field $\Phi$, 
the restriction $\Phi|_{\nabla'}$ is merely a pre-Higgs field in general, 
for it does not need to satisfy $\Phi|_{\nabla'}\wedge \Phi|_{\nabla'}=0$, cf.\ Remark~\ref{rem:com.inher}. However, we have the following

\begin{proposition}
\label{prop:faces.Higgspoly}
Let $\Phi$ be a Higgs field and $F\leq\nabla(\Phi)$ be a face. 
Then the restriction $\Phi|_F$ is a Higgs field, too.
In particular, the $\C$-linear pre-Higgs field $\phi^v$ arising from a 
\kbox{vertex $v\in\nabla(\Phi)$} via 
$\Phi|_v=\Phi^v=\cohikks^{v}\otimes\chi^v$ 
is an honest $\C$-linear Higgs field of degree $v$.
\end{proposition}

\begin{proof}
Let $a\in N$ be an integral vector defining the face $F$, i.e., 
$F=\{r\in M_\R\kst \langle r,a\rangle = 
\min \langle \nabla(\Phi),a\rangle\}$. 
From $\Phi=\sum_{r\in\nabla(\Phi)}\Phi^r$ we obtain 
$\Phi\wedge\Phi= 
\sum_{r,\,s\in \nabla(\Phi)}\Phi^r\wedge\Phi^s$ 
where the $(r,s)$-summand has degree $r+s\in M$.
Contracting the $M$-degrees via the linear map 
$\langle\kbb\,,a\rangle:M\to\Z$ exhibits the pairs 
$(r,s)\in F\times F$  exactly as those with minimal $\Z$-degree. 
Thus, $\Phi|_F\wedge\Phi|_F=(\Phi\wedge\Phi)|_{F\times F}=0$.
\end{proof}

\peeref{Subsections~\eqref{noPre} and~\eqref{subsec:HiggsPolyDP} will provide explicit examples of Higgs polytopes.}

\subsection{The Higgs range}
\label{sec:HiggsRange}
In order to see what kind of polytopes can arise for a given toric sheaf $\CE$, we call $r\in M$ {\em admissible for} $\CE$, if there exists a homogeneous pre-Higgs field $\Phi$ of degree $r$.

\begin{definition}\label{def-HiggsRange}
Let $\CE$ be a toric sheaf. The {\em Higgs range of} $\CE$ is the convex hull $\HR(\CE)$ in $M_\R$ defined by the admissible points. Moreover, for any $r\in\HR(\CE)$ we let $V_r(\CE)$ denote the complex vector space of maps $\phi:E\to E\otimes N_\C$ which are associated to some homogeneous pre-Higgs field $\Phi$ of degree $r$. 
\end{definition}

\peeref{We often think of $V_r(\CE)$ as a kind of multiplicity of the lattice point $r\in\HR(\CE)$.}

\medskip

From Proposition~\ref{prop:faces.Higgspoly} it follows that
the Higgs polytope $\nabla(\Phi)$ of every toric pre-Higgs field $\Phi$ on
$\CE$ has to be contained in $\HR(\CE)$. Moreover, $\HR(\CE)$
is a polytope itself by Proposition~\ref{prop-rangeBounded}, hence
there exists a maximal toric pre-Higgs field $\Phi$ satisfying
$\nabla(\Phi)=\HR(\CE)$. It is an immediate question
whether one can even find a true Higgs field $\Phi$ with this property.
The answer seems to be ``no'' or at least non-trivial
as it is indicated from the example of Subsection (\ref{toddler}).
Even more elementary is the question whether every admissible $r\in M$ 
does always allow a true homogeneous Higgs field $\Phi^r$ of degree $r$.

\medskip

\begin{proposition}\label{prop-rangeBounded}
Let $\CE$ be a toric sheaf over a complete toric variety. Then the Higgs range $\HR(\CE)$ is bounded, that is, it is a (possibly empty) convex polytope.
\end{proposition}

\begin{proof}
Recall from Theorem~\ref{th-describeHiggs} that
for all $\ell\in\Z$, we have at least
\begin{equation}\label{eq:HiggsCon}
\cohikks^{r}_s(E_\rho^{\ell})\subseteq
E_\rho^{\ell-1-\langle r,\rho\rangle}.
 \end{equation}
Denote by $N\in\N$ the maximum length of the filtrations
$E_\rho^\kbb$ for $\rho\in\Sigma(1)$. Then for each $\rho$
there exists an index $\ell(\rho)$ such that
$E_\rho^{\ell(\rho)}=E$, but $E_\rho^{\ell(\rho)+N}=0$.
In particular,
$$
\HR(\CE)\subseteq\{r\in M_\R\kst \langle r,\rho\rangle\geq -N
\;\mbox{for all }\rho\in\Sigma(1)\},
$$
where the latter set is bounded by completeness. Indeed, if one of these inequalities is violated,
say $\langle r,\rho\rangle< -N$,
then $-1-\langle r,\rho\rangle\geq N$, 
so that $\cohikks^{r}_s=0$ for every $s\in M$ by~\eqref{eq:HiggsCon}, and thus $\cohikks^{r}=0$.
\end{proof}

\peeref{
\begin{remark}
\label{rem:chernRoot}
In Remark~\ref{rem:klya.comp} we identified the jump loci
$\ku{u}(\sigma)\subset M/(M\cap\sigma^\bot)$ 
of the filtrations $E_\rho^\kbb$ as the equivariant Chern roots of
$\CE$. The proof of Proposition~\ref{prop-rangeBounded} then shows 
that the size of the Higgs range $\HR(\CE)$ is controlled 
by the stretching of these roots.
\end{remark}
}

\begin{example}\label{ex-rangeBoundedTang}
Recall from Example~\ref{ex-Klyachko}\,(iii) that the tangent sheaf $\CT_X$ 
is encoded by the filtrations $E_\rho^\kbb$ of $N_\C$ with 
$E_\rho^{1}=\spann(\rho)$, $\rho\in\Sigma(1)$, as the only nontrivial subspace. 
It follows that $N=2$ and $\ell(\rho)=0$ for all $\rho\in\Sigma(1)$. On the other hand, the fan of the projective plane $\P^2$ is the inner normal fan of the polytope $\Delta$ cut out by the 
equations $\langle m,\rho_0\rangle\geq1$ and 
$\langle m,\rho_{1,2}\rangle\geq0$.
As a result, the proof of Proposition~\ref{prop-rangeBounded} shows that
$\HR(\CT_{\PP^2})$ is contained in the polytope whose facets are at distance 
two from the origin and parallel to $\Delta$, see the 
red lines in the figure below. 
However, the true Higgs range $\HR(\CT_{\PP^2})$ is even smaller;
it is given by the yellow polytope. 
See Section~\ref{revP2} for the computation;
the result for $\HR$ can be found in Subsection~(\ref{linDeps}). 
Further Higgs ranges are computed in Section~\ref{sec:MInSur}.
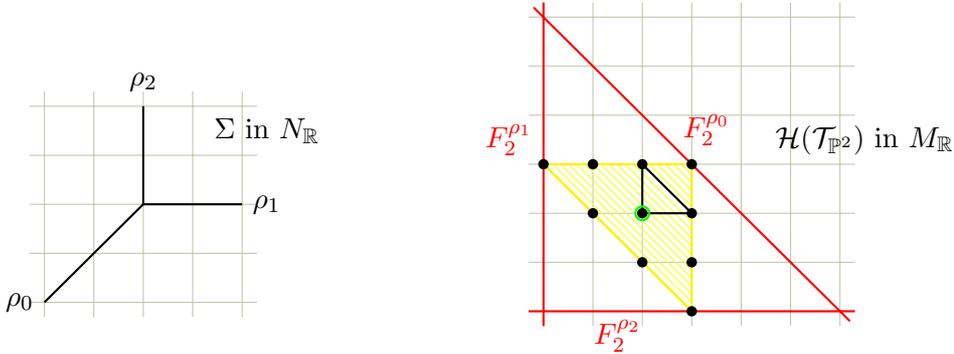
$$
\begin{tikzpicture}[scale=0.65]
\draw[color=oliwkowy!40] (-1.3,-0.3) grid (3.3,4.3);
\draw[thick,  color=black]
  (1,2) -- (3,2) (1,2) -- (1,4) (1,2) -- (-1,0);
\draw[thick, color=black]  (-1.5,0) node{$\rho_0$};
\draw[thick, color=black]  (3.5,2) node{$\rho_1$};
\draw[thick, color=black]  (1,4.5) node{$\rho_2$};
\draw[thick, color=black]  (3.5,3.5) node{$\Sigma$ in $N_\R$};
\draw[thick, color=black]  (1.0,-1.0) node{};
\end{tikzpicture}
\hspace{5em}
\begin{tikzpicture}[scale=0.65]
\draw[color=oliwkowy!40] (-2.3,-2.3) grid (4.3,4.3);
\fill[pattern color=yellow!90, pattern=north west lines]
 (-2,1) -- (1,1) -- (1,-2) -- cycle;
\draw[thick,  color=yellow]
  (-2,1) -- (1,1) -- (1,-2) -- cycle;
\draw[thick,  color=black]
  (1,0) -- (0,1);
\draw[thick,  color=red]
  (-2,4.3) -- (-2,-2.3) (-2.3,-2) -- (4.3,-2) (-2.2,4.2) -- (4.2,-2.2);
\draw[thick,  color=black] 
  (0,0) -- (0,1) (0,0) -- (1,0);
\fill[thick,  color=black] (0,0) circle (3pt);
\draw[thick, color=red]  (1.3,1.7) node{$F^{\rho_0}_2$};
\draw[thick, color=red]  (-2.7,1.5) node{$F^{\rho_1}_2$};
\draw[thick, color=red]  (-0.5,-2.5) node{$F^{\rho_2}_2$};
\fill[thick, color=black]
  (-1,0) circle (3pt) (-1,1) circle (3pt) (0,-1) circle (3pt)
  (1,-1) circle (3pt) (0,1) circle (3pt) (1,0) circle (3pt)
  (-2,1) circle (3pt) (1,-2) circle (3pt) (1,1) circle (3pt);
\draw[thick, color=green]
  (0,0) circle (4pt);
\draw[thick, color=black]  (4.5,1.5) node{$\HR(\CT_{\PP^2})$ in $M_\R$};
\end{tikzpicture}
$$
\end{example}

\section{\peeref{Trace-free} $\CT_X$-Higgs fields on $\PP^2$}
\label{revP2}
The computation of $\HR(\CE)$ and $\{V_r(\CE)\}_{r\in\HR(\CE)}$ for the 
intrinsic case $\CE=\CT_X$ will occupy us for the remainder of this paper. 
For simplicity, we write $\HR(X)$ and $V_r(X)$ in this case,
see Definition~\ref{def-HiggsRange}, and speak simply of (pre-)Higgs fields on $X$. 
Moreover, we will restrict to \peeref{trace-free} Higgs fields from now -- the reason is
that we can decompose any Higgs field into a \peeref{trace-free} one
and a vector field. \peeref{The latter were already computed 
in Example~\ref{ex-KlyachkoGlob}, cf.\ also Figure~\ref{fig:HiggsFieldsP2}.}
The corresponding subspaces we will denote by
$V^0_r(\PP^2)\subseteq V_r(\PP^2)$.
In the present section we focus on $\PP^2$ to illustrate the key ideas. 
\peeref{We also recover some results from the article~\cite{rayancoHiggs2} 
which also studies Higgs fields for the tangent sheaf of $\PP^2$
albeit from a non-toric point of view.}

\subsection{Encoding endomorphisms}
\label{encodeEndo}
In the particular case of $\PP^2$, we try to keep the
symmetry by understanding
$$
N=\Z^3/\ku{1}\cdot \Z
\hspace{1em}\mbox{and}\hspace{1em}
M=\ku{1}^\bot=\{r\in\Z^3\kst r_0+r_1+r_2=0\}\subseteq\Z^3.
$$
Thus, any $\cohikks:=\cohikks^r_s:N_\C\to N_\C$ 
becomes a map $\C^3\to\C^3$ sending
$\ku{1}$ into $\spann(\ku{1})$. As a result, $\cohikks$ is represented by
a $(3\times3)$-matrix 
\peeref{with equal row sums, i.e.}
$$
\textstyle
\tcohikks=\left(\begin{array}{ccc}
c_{00} & c_{01} & c_{02}\\
c_{10} & c_{11} & c_{12}\\
c_{20} & c_{21} & c_{22}
\end{array}\right)
\hspace{0.8em}\mbox{with}\hspace{0.8em}
\sum_{j=0}^2c_{ij} \hspace{0.5em}
\mbox{being independent of the row \peeref{number} $i$}.
$$
Altering $\tcohikks$ into
$\tcohikks+(\ku{a}\;\ku{b}\;\ku{c})$, i.e.~adding \peeref{a matrix consisting of
3 equal rows $(a\,b\,c)$}, does not change $\cohikks$. 
Hence, we obtain a canonical representative (also called $\cohikks$) 
by asking for $c_{11}=c_{22}=c_{33}=0$.
We will sometimes use the isomorphism 
$\Z^2\stackrel{\sim}{\to}N=\Z^3/\ku{1}\cdot \Z$ and its inverse
$\Z^3\surj N \stackrel{\sim}{\to}\Z^2$
given by the matrices
$$
\textstyle
\left(\begin{array}{cc} 0 & 0 \\ 1 & 0 \\ 0 & 1 \end{array}\right)
\hspace{1em}\mbox{and}\hspace{1em}
\left(\begin{array}{ccc} -1 & 1 & 0 \\ -1 & 0 & 1 \end{array}\right),
$$
respectively, in order to view $\cohikks$ or $\tcohikks$ as a linear map
$\C^2\to\C^2$ \peeref{given by}
$$
\textstyle
\left(\begin{array}{ccc} -1 & 1 & 0 \\ -1 & 0 & 1 \end{array}\right)
\cdot
\left(\begin{array}{ccc}
c_{00} & c_{01} & c_{02}\\
c_{10} & c_{11} & c_{12}\\
c_{20} & c_{21} & c_{22}
\end{array}\right)
\cdot
\left(\begin{array}{cc} 0 & 0 \\ 1 & 0 \\ 0 & 1 \end{array}\right)
=
\left(\begin{array}{cc} c_{11}-c_{01} & c_{12}-c_{02} \\
                        c_{21}-c_{01} & c_{22}-c_{02}
\end{array}\right).
$$
This yields $\trace(\cohikks)=(c_{00}+c_{11}+c_{22})-\sum_{j=0}^2c_{ij}$
(for each $i=0,1,2$). In particular, the normal form of $\cohikks$ equals
$$
\textstyle
\cohikks=
\left(\begin{array}{ccc}
0 & c_{01} & c_{02}\\
c_{10} & 0 & c_{12}\\
c_{20} & c_{21} & 0
\end{array}\right)
\hspace{1em}\mbox{with}\hspace{1em}
c_{01}+c_{02}=c_{10}+c_{12}=c_{20}+c_{21}=-\trace(\cohikks).
$$
Note that $\trace(\cohikks)$ does not refer to the 
literal interpretation as the trace of the representing
$(3\times 3)$-matrix, but of $\cohikks\in\gEnd(N)$.
As mentioned above, we will focus on \kbox{\peeref{trace-free}} endomorphisms. 
They can be written as
$$
\textstyle
\cohikksT=
\left(\begin{array}{ccc}
0 & x & -x \\ -y & 0 & y \\ z & -z & 0
\end{array}\right)
=x A_0 + y A_1 + z A_2
\leadsto
\cohikksT=
\left(\begin{array}{cc}
-x & x+y \\ -(x+z) & x
\end{array}\right)
$$
with
$$
A_0=\left(\begin{array}{rrr} 
0 & 1 & -1 \\ 0 & 0 & 0 \\ 0 & 0 & 0
\end{array}\right),\hspace{1em}
A_1=\left(\begin{array}{rrr}
0 & 0 & 0 \\ -1 & 0 & 1 \\ 0 & 0 & 0
\end{array}\right),\hspace{1em}
A_2=\left(\begin{array}{rrr}
0 & 0 & 0 \\ 0 & 0 & 0 \\ 1 & -1 & 0
\end{array}\right).
$$
The determinant (as an endomorphism of $N_\C$) 
is \kbox{$\det(\cohikksT)=xy+yz+zx$}.

\subsection{From filtrations to facets}
\label{filtFac}
\peeref{Recall from Example~\ref{ex-Klyachko}\,(iii) and
Example~\ref{ex-rangeBoundedTang} that
$E_\rho^{1}=\spann(\rho)$ with $\rho\in\Sigma(1)$
are the only non-trivial parts of the filtrations encoding $\CT_{\PP^2}$.
Moreover, as we have already used in the proof of 
Proposition~\ref{prop-rangeBounded}, a pre-Higgs field satisfies
$\cohikks^{r}_s(E_\rho^{\ell})\subseteq
E_\rho^{\ell-1-\langle r,\rho\rangle}$ for all $\ell\in\Z$. This can be sharpened to
$\cohikks^{r}_s(E_\rho^{\ell})\subseteq
E_\rho^{\ell-\langle r,\rho\rangle}$ if
$s\in\rho^\bot$, cf.\ Theorem~\ref{th-describeHiggs}.}

\medskip

Next take $\rho\in\Sigma(1)$ and $c\in\Z_{\geq 0}$. We consider the
affine hyperplanes
$$
F^\rho_c:=\big[\langle \kbb,-\rho\rangle=c\big]
\hspace{1em}\mbox{and}\hspace{1em}
F^\rho_{\geq c}:=\big[\langle\kbb,-\rho\rangle\geq c\big].
$$
at lattice distance $c$ from the origin. 
They are parallel to the corresponding facets of the polytope
$\Delta=\conv\{[0,0],\; [1,0],\; [0,1]\}$ from
Example~\ref{ex-rangeBoundedTang}. 
The second notion $F^\rho_{\geq c}$ points
to the outside area, i.e., beyond $F^\rho_c$.
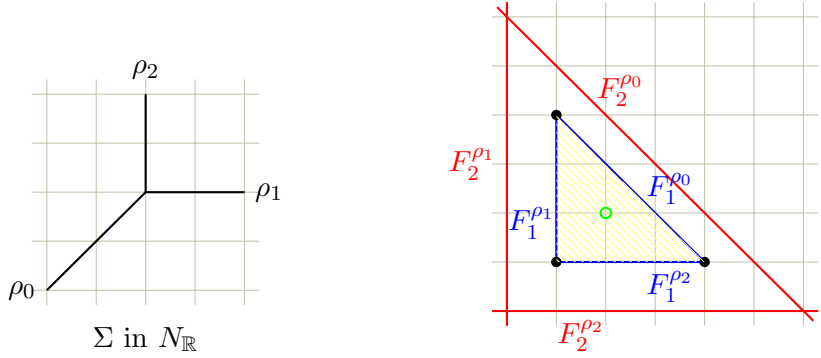
$$
\begin{tikzpicture}[scale=0.65]
\draw[color=oliwkowy!40] (-1.3,-0.3) grid (3.3,4.3);
\draw[thick, color=black]  (1.0,-1.0) node{$\Sigma$ in $N_\R$};
\draw[thick,  color=black]
  (1,2) -- (3,2) (1,2) -- (1,4) (1,2) -- (-1,0);
\draw[thick, color=black]  (-1.5,0) node{$\rho_0$};
\draw[thick, color=black]  (3.5,2) node{$\rho_1$};
\draw[thick, color=black]  (1,4.5) node{$\rho_2$};
\end{tikzpicture}
\hspace{5em}
\begin{tikzpicture}[scale=0.65]
\draw[color=oliwkowy!40] (-2.3,-2.3) grid (4.3,4.3);
\draw[thick,  color=blue]
  (-1,-1) -- (2,-1) -- (-1,2) -- cycle;
\draw[thick,  color=red]
  (-2,4.3) -- (-2,-2.3) (-2.3,-2) -- (4.3,-2) (-2.2,4.2) -- (4.2,-2.2);
\fill[thick,  color=black] 
  (-1,-1) circle (3pt) (2,-1) circle (3pt) (-1,2) circle (3pt);
\fill[pattern color=yellow!50, pattern=north west lines]
  (-1,-1) -- (2,-1) -- (-1,2) -- cycle;
\draw[thick,  color=green] (0,0) circle (3pt);
\draw[thick, color=red]  (0.3,2.5) node{$F^{\rho_0}_2$};
\draw[thick, color=red]  (-2.7,1) node{$F^{\rho_1}_2$};
\draw[thick, color=red]  (-0.5,-2.5) node{$F^{\rho_2}_2$};
\draw[thick, color=blue]  (1.3,0.5) node{$F^{\rho_0}_1$};
\draw[thick, color=blue]  (-1.5,-0.2) node{$F^{\rho_1}_1$};
\draw[thick, color=blue]  (1.3,-1.5) node{$F^{\rho_2}_1$};
\end{tikzpicture}
$$
The more ``$\rho$-outside'' the degrees $r$ are,
i.e., the larger $c$ with $r\in F^\rho_{\geq c}$, 
the larger have to be the $\rho$-jumps $j$
with $\cohikks^r_s(E_\rho^\ell)\subseteq E_\rho^{\ell+j}$, 
i.e., the more restricted is $\cohikks^r$. To make this precise
we introduce the following notation: An endomorphism
$\cohikks\in\gEnd(E)$ is said to belong to the classes
\vspace{1ex}

\begin{enumerate}
\item[(i)$_\rho$]
if \kbox{$\cohikks(\rho)\in\spann(\rho)$}, and
\vspace{1.0ex}
\item[(ii)$_\rho$]
if \kbox{$\cohikks(E)\subseteq\spann(\rho)\subseteq\ker(\cohikks)$}.
Note that the latter implies nilpotency.
\vspace{1.0ex}
\end{enumerate}

Obviously, [(ii)$_\rho$ $\then$ (i)$_\rho$]. In the language of
Subsection (\ref{encodeEndo}), these conditions translate 
for \peeref{trace-free} endomorphisms $\cohikks$ as follows
(here explained for $\rho=\rho_0$):
\begin{enumerate} 
\item[(i)$_0$]
$\cohikks$ is a linear combination 
$$
\cohikks=
\left(\begin{array}{ccc}
0 & x &-x\\
z & 0 &-z\\
z &-z & 0
\end{array}\right)
= xA_0+z(A_2-A_1), \mbox{ and}
$$
\item[(ii)$_0$]
$
\cohikks=
\left(\begin{array}{ccc}
0 & x &-x\\
0 & 0 & 0\\
0 & 0 & 0
\end{array}\right) = x\cdot A_0
$
is nilpotent.
\end{enumerate}

\begin{lemma}
\label{lem-rotBlau}
Let $r,s\in M$. Then $\cohikks^r_s$ satisfies the general
Higgs condition 
$\cohikks^{r}_s(E_\rho^{\kbb})\subseteq
E_\rho^{\kbb-1-\langle r,\rho\rangle}$
if and only if
$
\renewcommand{\arraystretch}{1.2}
\left\{\begin{array}{@{}ll@{}}
r\in F^\rho_{\geq 3} & \then \cohikks^{r}_s=0\\
r\in F^\rho_{2} & \then \cohikks^{r}_s\in {\rm (ii)}_\rho\\
r\in F^\rho_{1} & \then \cohikks^{r}_s\in {\rm (i)}_\rho.
\end{array}\right\}$
Moreover, if $s\in\rho^\bot$, then the associated stronger condition arises
from replacing $F^\rho_{\geq 3}$ by $F^\rho_{\geq 2}$ and
$F^\rho_{i}$ by $F^\rho_{i-1}$ for $i=1,2$.
\end{lemma}

The proof is straightforward.

\medskip

In the following figures we will indicate the conditions 
(i)$_\rho$ and (ii)$_\rho$ by the colors blue and red, respectively.
Moreover, we put black (or green) dots on all lattice points $r\in M$ where
a non-vanishing $\cohikks^{r}_s$ is (still) possible. 
For general $s\in M$, we start with a blue $3\Delta$ and a red $6\Delta$
(shifted into  central position).
For $s\in\rho^\bot$, the $\rho$-facets will be shifted towards the
origin yielding blue $2\Delta$ and red $5\Delta$ in different positions.

\medskip

Note that for linearily independent $\rho,\rho'$ the intersection
${\rm (i)}_\rho\cap{\rm (ii)}_{\rho'}$ leads to $\cohikks^{r}_s=0$.
Hence, we can exclude red-red and red-blue intersections.
The blue-blue intersection ${\rm (i)}_{\nu-1}\cap{\rm (i)}_{\nu+1}$
leads to the unique $\cohikks=A_{\nu-1}+A_{\nu+1}-A_\nu$
($\nu\in\Z/3\Z$).
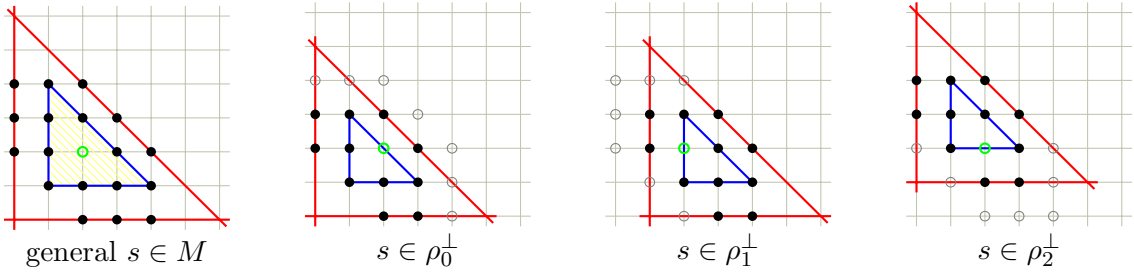
$$
\newcommand{\abstZ}{\hspace{2.5em}}
\begin{tikzpicture}[scale=0.45]
\draw[color=oliwkowy!40] (-2.3,-2.3) grid (4.3,4.3);
\fill[pattern color=yellow!50, pattern=north west lines]
  (-1,-1) -- (2,-1) -- (-1,2) -- cycle;
\draw[thick,  color=blue]
  (-1,-1) -- (2,-1) -- (-1,2) -- cycle;
\draw[thick,  color=red]
  (-2,4.3) -- (-2,-2.3) (-2.3,-2) -- (4.3,-2) (-2.2,4.2) -- (4.2,-2.2);
\fill[thick,  color=black]
  (-1,-1) circle (4pt) (2,-1) circle (4pt) (-1,2) circle (4pt);
\fill[thick,  color=black]
  (-1,0) circle (4pt) (-1,1) circle (4pt) 
  (0,-1) circle (4pt) (1,-1) circle (4pt)
  (0,1) circle (4pt) (1,0) circle (4pt);
\fill[thick,  color=black]
  (-2,0) circle (4pt) (-2,1) circle (4pt) (-2,2) circle (4pt)
  (0,-2) circle (4pt) (1,-2) circle (4pt) (2,-2) circle (4pt)
  (0,2) circle (4pt) (1,1) circle (4pt) (2,0) circle (4pt);
\draw[thick,  color=black]  (1.0,-3.0) node{general $s\in M$};
\draw[thick,  color=green] (0,0) circle (4pt);
\end{tikzpicture}
\abstZ
\begin{tikzpicture}[scale=0.45]
\draw[color=oliwkowy!40] (-2.3,-2.3) grid (4.3,4.3);
\draw[thick,  color=blue]
  (-1,-1) -- (1,-1) -- (-1,1) -- cycle;
\draw[thick,  color=red]
  (-2,3.3) -- (-2,-2.3) (-2.3,-2) -- (3.3,-2) (-2.2,3.2) -- (3.2,-2.2);
\draw[thin, color=gray]
  (-2,2) circle (4pt) (-1,2) circle (4pt) (0,2) circle (4pt)
  (1,1) circle (4pt) (2,0) circle (4pt) (2,-1) circle (4pt) 
  (2,-2) circle (4pt);
\fill[thick,  color=black]
  (-1,-1) circle (4pt) (-1,0) circle (4pt) (-1,1) circle (4pt) 
  (0,-1) circle (4pt) (1,-1) circle (4pt) (0,1) circle (4pt) 
  (1,0) circle (4pt) (-2,0) circle (4pt) (-2,1) circle (4pt) 
  (0,-2) circle (4pt) (1,-2) circle (4pt);
\draw[thick,  color=black]  (1.0,-3.0) node{$s\in \rho_0^\bot$};
\draw[thick,  color=green] (0,0) circle (4pt);
\end{tikzpicture}
\abstZ
\begin{tikzpicture}[scale=0.45]
\draw[color=oliwkowy!40] (-2.3,-2.3) grid (4.3,4.3);
\draw[thick,  color=blue]
  (0,-1) -- (2,-1) -- (0,1) -- cycle;
\draw[thick,  color=red]
  (-1,3.3) -- (-1,-2.3) (-1.3,-2) -- (4.3,-2) (-1.2,3.2) -- (4.2,-2.2);
\draw[thin, color=gray]
  (-2,0) circle (4pt) (-2,1) circle (4pt) (-2,2) circle (4pt)
  (-1,2) circle (4pt) (0,2) circle (4pt) (-1,-1) circle (4pt) 
  (0,-2) circle (4pt);
\fill[thick,  color=black]
  (2,-1) circle (4pt) (-1,0) circle (4pt) (-1,1) circle (4pt) 
  (0,-1) circle (4pt) (1,-1) circle (4pt) (0,1) circle (4pt) 
  (1,0) circle (4pt) (1,-2) circle (4pt) (2,-2) circle (4pt)
  (1,1) circle (4pt) (2,0) circle (4pt);
\draw[thick,  color=black]  (1.0,-3.0) node{$s\in \rho_1^\bot$};
\draw[thick,  color=green] (0,0) circle (4pt);
\end{tikzpicture}
\abstZ
\begin{tikzpicture}[scale=0.45]
\draw[color=oliwkowy!40] (-2.3,-2.3) grid (4.3,4.3);
\draw[thick,  color=blue]
  (-1,0) -- (1,0) -- (-1,2) -- cycle;
\draw[thick,  color=red]
  (-2,4.3) -- (-2,-1.3) (-2.3,-1) -- (3.3,-1) (-2.2,4.2) -- (3.2,-1.2);
\draw[thin, color=gray]
  (0,-2) circle (4pt) (1,-2) circle (4pt) (2,-2) circle (4pt)
  (2,-1) circle (4pt) (2,0) circle (4pt) (-1,-1) circle (4pt)
  (-2,0) circle (4pt);
\fill[thick,  color=black]
  (-1,2) circle (4pt) (-1,0) circle (4pt) (-1,1) circle (4pt) 
  (0,-1) circle (4pt) (1,-1) circle (4pt) (0,1) circle (4pt) 
  (1,0) circle (4pt) (-2,1) circle (4pt) (-2,2) circle (4pt)
  (0,2) circle (4pt) (1,1) circle (4pt);
\draw[thick,  color=black]  (1.0,-3.0) node{$s\in \rho_2^\bot$};
\draw[thick,  color=green] (0,0) circle (4pt);
\end{tikzpicture}
$$

\subsection{Linear dependence on $s$}
\label{linDeps}
For a fixed $r\in M$, the endomorphisms $\cohikks^{r}_s$ do linearily depend
on $s\in M$. For the present $\PP^2$ example, we choose
$s^0:=[0,1,-1]$, $s^1:=[-1,0,1]$ and $s^2:=[1,-1,0]$ being contained in
$\rho_0^\bot$, $\rho_1^\bot$, and  $\rho_2^\bot$, respectively.
Denoting $\cohikks_i:=\cohikks^{r}_{s^i}$, the relation
$s^0+s^1+s^2=0$ implies $\cohikks_0+\cohikks_1+\cohikks_2=0$.
Moreover, since any two of $\{\cohikks_0,\cohikks_1,\cohikks_2\}$ span
$\{\cohikks_s\kst s\in M\}$, the vanishing of two 
$\cohikks_i$ implies the vanishing of $\cohikks^{r}_{s}$ for all
$s\in M$. This leads to the following improvement of the previous
figures:
$$
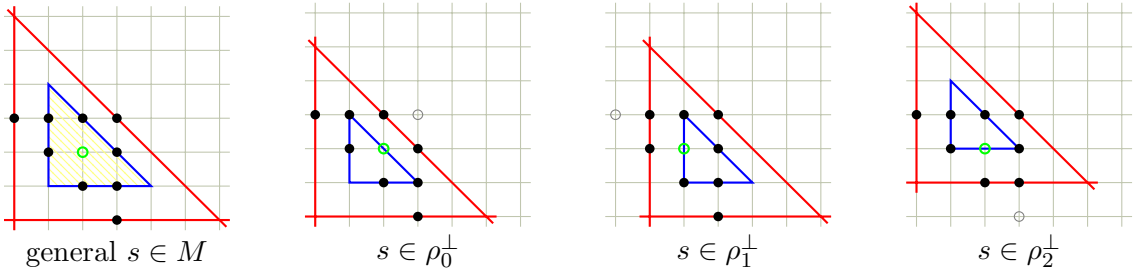
\newcommand{\abstZ}{\hspace{2.5em}}
\begin{tikzpicture}[scale=0.45]
\draw[color=oliwkowy!40] (-2.3,-2.3) grid (4.3,4.3);
\fill[pattern color=yellow!50, pattern=north west lines]
  (-1,-1) -- (2,-1) -- (-1,2) -- cycle;
\draw[thick,  color=blue]
  (-1,-1) -- (2,-1) -- (-1,2) -- cycle;
\draw[thick,  color=red]
  (-2,4.3) -- (-2,-2.3) (-2.3,-2) -- (4.3,-2) (-2.2,4.2) -- (4.2,-2.2);
\fill[thick,  color=black]
  (-1,0) circle (4pt) (-1,1) circle (4pt) (0,-1) circle (4pt) 
  (1,-1) circle (4pt) (0,1) circle (4pt) (1,0) circle (4pt) 
  (-2,1) circle (4pt) (1,-2) circle (4pt) (1,1) circle (4pt);
\draw[thick,  color=black]  (1.0,-3.0) node{general $s\in M$};
\draw[thick,  color=green] (0,0) circle (4pt);
\end{tikzpicture}
\abstZ
\begin{tikzpicture}[scale=0.45]
\draw[color=oliwkowy!40] (-2.3,-2.3) grid (4.3,4.3);
\draw[thick,  color=blue]
  (-1,-1) -- (1,-1) -- (-1,1) -- cycle;
\draw[thick,  color=red]
  (-2,3.3) -- (-2,-2.3) (-2.3,-2) -- (3.3,-2) (-2.2,3.2) -- (3.2,-2.2);
\draw[thin, color=gray]
  (1,1) circle (4pt);
\fill[thick,  color=black]
  (-1,0) circle (4pt) (-1,1) circle (4pt) (0,-1) circle (4pt) 
  (1,-1) circle (4pt) (0,1) circle (4pt) (1,0) circle (4pt) 
  (-2,1) circle (4pt) (1,-2) circle (4pt);
\draw[thick,  color=black]  (1.0,-3.0) node{$s\in \rho_0^\bot$};
\draw[thick,  color=green] (0,0) circle (4pt);
\end{tikzpicture}
\abstZ
\begin{tikzpicture}[scale=0.45]
\draw[color=oliwkowy!40] (-2.3,-2.3) grid (4.3,4.3);
\draw[thick,  color=blue]
  (0,-1) -- (2,-1) -- (0,1) -- cycle;
\draw[thick,  color=red]
  (-1,3.3) -- (-1,-2.3) (-1.3,-2) -- (4.3,-2) (-1.2,3.2) -- (4.2,-2.2);
\draw[thin, color=gray]
  (-2,1) circle (4pt);
\fill[thick,  color=black]
  (-1,0) circle (4pt) (-1,1) circle (4pt) (0,-1) circle (4pt) 
  (1,-1) circle (4pt) (0,1) circle (4pt) (1,0) circle (4pt) 
  (1,-2) circle (4pt) (1,1) circle (4pt);
\draw[thick,  color=black]  (1.0,-3.0) node{$s\in \rho_1^\bot$};
\draw[thick,  color=green] (0,0) circle (4pt);
\end{tikzpicture}
\abstZ
\begin{tikzpicture}[scale=0.45]
\draw[color=oliwkowy!40] (-2.3,-2.3) grid (4.3,4.3);
\draw[thick,  color=blue]
  (-1,0) -- (1,0) -- (-1,2) -- cycle;
\draw[thick,  color=red]
  (-2,4.3) -- (-2,-1.3) (-2.3,-1) -- (3.3,-1) (-2.2,4.2) -- (3.2,-1.2);
\draw[thin, color=gray]
  (1,-2) circle (4pt);
\fill[thick,  color=black]
  (-1,0) circle (4pt) (-1,1) circle (4pt) (0,-1) circle (4pt) 
  (1,-1) circle (4pt) (0,1) circle (4pt) (1,0) circle (4pt) 
  (-2,1) circle (4pt) (1,1) circle (4pt);
\draw[thick,  color=black]  (1.0,-3.0) node{$s\in \rho_2^\bot$};
\draw[thick,  color=green] (0,0) circle (4pt);
\end{tikzpicture}
$$
Finally, this yields the

\begin{theorem}
\label{th-HiggsRangePb}
The Higgs range $\HR({\PP^2})$
equals the convex hull of the points $r^0=[-2,1,1]$, $r^1=[1,-2,1]$, and 
$r^2=[1,1,-2]$. Moreover, for each $\nu\in\Z/3\Z$ we have
$\cohikks^{r^\nu}_{s^\nu}=0$.
The dimensions of the vector space $V^0_r(\PP^2)$ equal 
$1$ if $r$ is a vertex of $\HR({\PP^2})$, equal $2$ if $r$ is among the six
remaining lattice points on the boundary, and equal\peeref{s} $3$ for $r=0$.
\vspace{-2ex}
\end{theorem}

$$
\begin{tikzpicture}[scale=0.45]
\draw[color=oliwkowy!40] (-2.3,-2.3) grid (4.3,4.3);
\fill[pattern color=green!50, pattern=north east lines]
  (-2,1) -- (1,1) -- (1,-2) -- cycle;
\draw[thin,  color=blue]
  (-1,-1) -- (2,-1) -- (-1,2) -- cycle;
\draw[thin,  color=red]
  (-2,4.3) -- (-2,-2.3) (-2.3,-2) -- (4.3,-2) (-2.2,4.2) -- (4.2,-2.2);
\fill[thick,  color=black]
  (-1,0) circle (4pt) (-1,1) circle (4pt) (0,-1) circle (4pt)
  (1,-1) circle (4pt) (0,1) circle (4pt) (1,0) circle (4pt);
\draw[thick,  color=black] 
  (-2,1) circle (4pt) (1,-2) circle (4pt) (1,1) circle (4pt);
\draw[thick,  color=black]  (1.0,-3.8) node{The Higgs range $\HR(\PP^2)$};
\draw[thick,  color=green] (0,0) circle (4pt);
\draw[thick,  color=black]
  (-2,1) -- (1,-2) -- (1,1) -- cycle;
\draw[thick,  color=black]  
  (-2.7,1.5) node{$r^1$} (1.6,1.5) node{$r^0$} (1.6,-2.5) node{$r^2$};
\end{tikzpicture}
$$

\begin{proof}
\peeref{So far we have seen that $\HR(\PP^2)$ is contained in the asserted convex
hull, and we do also know that $\cohikks^{r^\nu}_{s^\nu}=0$.}
It remains to check the dimensions of $V_r^0(\PP^2)$ -- but this will be done in
Subsection~(\ref{endoKnot}).
\end{proof}

\subsection{Analysing the endomorphisms for each \peeref{degree}}
\label{endoKnot}
We have three type of degrees $r\in M$. Let us take a closer look to all of
them and their associated $\cohikks^{r}_{s}$.
Recall from Subsection~(\ref{linDeps}) 
that we abreviate $\cohikks_\nu:=\cohikks^r_{s^\nu}$ where
$s^0=[0,1,-1]$, $s^1=[-1,0,1]$, and $s^2=[1,-1,0]$.
Moreover, we will use
for a general $s=[s_0,s_1,s_2]\in\ku{1}^\bot=M$ the representation
$s=[s_0,s_1,s_2]=s_1\cdot s^0-s_0\cdot s^1=s_0\cdot s^2-s_2\cdot s^0$.
Further, we keep the 
\peeref{trace-free} endomorphisms $A_0$, $A_1$, and $A_2$ 
from Subsection (\ref{encodeEndo}).
The degrees which are candidate for carrying
pre-Higgs fields can be subdivided into three classes.
\vspace{1ex}
\begin{enumerate}
\item[(i)]
{\em $r\in\{r^0,r^1,r^2\}$ is a vertex of $\HR$.}
Let us assume that $r=r^0$. 
\\[0.5ex]
For all $s\in M$, $\cohikks^{r^0}_{s}$ is of class (ii)$_0$, cf.\ 
Subsection (\ref{filtFac}). 
Moreover, $\cohikks^{r^0}_{0}=0$
implies that $
\cohikks^{r^0}_{s} = 
-s_0\cdot \cohikks^{r^0}_{1} = s_0\cdot \cohikks^{r^0}_{2}
$ for $s=[s_0,s_1,s_2]=s_1\cdot s^0-s_0\cdot s^1\in M$.
Altogether, this means that
$$
\cohikks^{r^0}_{s} = 
\left(\begin{array}{ccc}
0 & c_0s_0 &-c_0s_0\\
0 & 0 & 0\\ 
0 & 0 & 0
\end{array}\right)
= c_0\cdot\langle s,\rho_0\rangle\cdot A_0,\hspace{0.5em}
\mbox{i.e.,}\;
\cohikks^{r^0}=c_0\cdot A_0\otimes \rho_0
$$
for some parameter $c_0\in\C$. In other words,
$\{A_0\otimes \rho_0\}$ is a $\C$-basis for 
all possible $\cohikks^{[-2,1,1]}$.
\vspace{1ex}
\item[(ii)]
{\em $r\in M\cap\partial\HR$, but $r$ is not a vertex of $\HR$.}
Let us assume that $r=[-1,1,0]$. 
\\[0.5ex]
In the coordinates of the previous figures, $r$ equals $[1,0]$,
i.e., in the figures for $\rho_0^\bot$, $\rho_1^\bot$, and $\rho_2^\bot$
$r$ sits on the $0$-red line, the $0$-blue line, and the intersection
of the $0$-blue and the $2$-blue lines, respectively. 
This information translates
into \kbox{$\cohikks^{[-1,1,0]}_0\in {\rm (ii)}_0$, 
$\cohikks^{[-1,1,0]}_1\in {\rm (i)}_0$,
and $\cohikks^{[-1,1,0]}_2\in {\rm (i)}_0\cap{\rm (i)}_2$}.
\\
These three classes equal $\spann(A_0)$, 
$\spann(A_0,\, A_2-A_1)$, and $\spann(A_0-A_1+A_2)$, 
respectively. Thus, the relation
$\cohikks_0+\cohikks_1+\cohikks_2=0$ leads to the following basis
for the vector space of possible \peeref{trace-free} $\cohikks^{[-1,1,0]}$:
$$
\{A_0\otimes\rho_2,\; (A_0-A_1+A_2)\otimes \rho_0\}.
$$
\vspace{-2ex}
\item[(iii)]
{\em $r=0$ is the origin.}
\\[0.5ex]
Here we know that $\cohikks^0_\nu\in {\rm (i)}_\nu$ for $\nu=0,1,2$.
The \peeref{trace-free} part of these classes is
$\,\spann(A_0,\, A_2-A_1)$, $\,\spann(A_1,\, A_0-A_2)$, and
$\,\spann(A_2,\, A_1-A_0)$,
respectively. This implies that
$$
\begin{array}{rcl}
\cohikks_\nu^0 &=&
(c_{\nu-1}-c_{\nu+1})A_\nu + c_\nu(A_{\nu-1}-A_{\nu+1})\\
&=&
(c_{\nu-1}A_\nu + c_\nu A_{\nu-1}) -
(c_{\nu+1}A_\nu + c_\nu A_{\nu+1})
\end{array}
$$
for $\nu\in\Z/3\Z$ and some $c\in\C^3$.
This leads to the following basis
for the vector space of all possible \peeref{trace-free} $\cohikks^0$:
$$
\{(A_2\otimes\rho_1+A_1\otimes\rho_2),\; 
(A_0\otimes\rho_2+A_2\otimes\rho_0),\; 
(A_1\otimes\rho_0+A_0\otimes\rho_1)\}.
\vspace{1ex}
$$
\end{enumerate}

The dimension of the space of \peeref{trace-free} Higgs fields on $\PP2$ equals 
$(3\times 1)+(6\times 2) + (1\times 3)=18$. The full vector space of pre-Higgs 
fields is obtained by adding the eight dimensional space of vector fields, 
cf.~\ref{ex-KlyachkoGlob} which are of the form $\id_N\otimes\rho_i$. 

\medskip

The details are summarised in Figure~\ref{fig:HiggsFieldsP2}. For a better orientation we have kept the 
Higgs range polytope $\HR$ (indicated in black) and the original reflexive 
polytope $3\Delta$ (indicated in blue). As before, the origin is visible as 
a green circle. 

\begin{figure}[ht]
\newcommand{\coH}[2]{{\begin{minipage}{#1em}{$\ks #2$}\end{minipage}}}
\newcommand{\cid}[1]{{\color{red}s_#1\id_N}}
\newcommand{\aid}[1]{\mbox{ \tiny and }}
\begin{tikzpicture}[x=3cm, y=1.5cm]
\draw[thin,  color=blue]
  (-1,-1) -- (2,-1) -- (-1,2) -- cycle;
\draw[thick,  color=blue]
  (-0.6,1.8) node{$3\Delta$};
\draw[thin,  color=gray]
  (-1,0) circle (2pt) (-1,1) circle (2pt) (0,-1) circle (2pt)
  (1,-1) circle (2pt) (0,1) circle (2pt) (1,0) circle (2pt);
\draw[thin,  color=black]
  (-2,1) circle (15pt) (1,-2) circle (15pt) (1,1) circle (15pt);
\draw[thick,color=black] 
 (0,0) node{\fbox{$\coH{5.3}{A_2\otimes\rho_1+A_1\otimes\rho_2\\
   A_0\otimes\rho_2+A_2\otimes\rho_0\\A_1\otimes\rho_0+A_0\otimes\rho_1\\\langle\{\id_N\otimes\rho_i\}_{i=1}^3\rangle}$}} 
 (1,0) node{\fbox{$\coH{6.1}{A_0\otimes\rho_2,\,\id_N\otimes\rho_0\\
   (A_0-A_1+A_2)\otimes\rho_0}$}} 
 (1,1) node{{$\coH{2.8}{A_0\otimes\rho_0}$}}
 (0,1) node{\fbox{$\coH{6.1}{A_0\otimes\rho_1,\,\id_N\otimes\rho_0\\
   (A_0+A_1-A_2)\otimes\rho_0}$}}
 (-1,1) node{\fbox{$\coH{6.1}{A_1\otimes\rho_0,\,\id_N\otimes\rho_1\\
   (A_0+A_1-A_2)\otimes\rho_1}$}}
 (-2,1) node{{$\coH{2.5}{A_1\otimes\rho_1}$}}
 (-1,0) node{\fbox{$\coH{6.7}{A_1\otimes\rho_2,\,\id_N\otimes\rho_1\\
   (-A_0+A_1+A_2)\otimes\rho_1}$}}
 (0,-1) node{\fbox{$\coH{6.7}{A_2\otimes\rho_1,\,\id_N\otimes\rho_2\\
   (-A_0+A_1+A_2)\otimes\rho_2}$}}
 (1,-2) node{{$\coH{2.5}{A_2\otimes\rho_2}$}}
 (1,-1) node{\fbox{$\coH{6.1}{A_2\otimes\rho_0,\,\id_N\otimes\rho_2\\
   (A_0-A_1+A_2)\otimes\rho_2}$}};
\draw[thick,  color=green] (0,0) circle (4pt);
\draw[thin,  color=gray]
  (-2,1) -- (1,-2) -- (1,1) -- cycle;
\end{tikzpicture}
\caption{\peeref{The pre-Higgs fields on $\P^2$. For each nontrivial degree the box respectively the circle at $r\in M$ contains a basis for the linear pre-Higgs fields of degree $r$. For instance, the homogeneous pre-Higgs fields of degree $(-2,1)$ are of the form $c_0A_1\otimes\rho_1\otimes\chi^{(-2,1)}$, $c_0\in\C$. To keep the symmetry we used a generating system for the degree $0$ pre-Higgs field of pure trace.}}\label{fig:HiggsFieldsP2}
\end{figure}
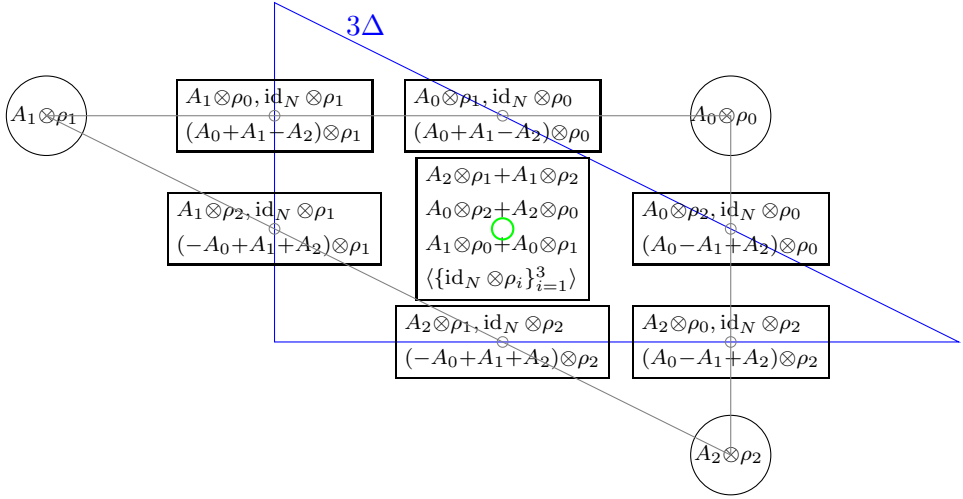

\subsection{From pre-Higgs to Higgs}
\label{noPre}
The commutators of the \peeref{trace-free} matrices $A_i$ ($i\in\Z/3\Z$)
can be expressed as
$$
[A_{i-1},A_{i+1}]=(A_{i-1}+A_{i+1})-A_i
\hspace{2em}(i\in\Z/3\Z).
$$
In Subsection (\ref{endoKnot}) we got an 18-dimensional space of 
trace-free pre-Higgs fields on $\PP^2$. 
Using {\sc Singular}~\cite{singular}, 
we have incorporated the commutator vanishing
$[\coHikks_s,\coHikks_t]=0$ for all $s,t\in M$. This leads
to an ideal $I$ in 18 variables with 100 generators (increasing to 435 
generators after calculating a dp-Gr\"obner basis). 
The dimension of $V(I)\subset\C^{18}$ is 8, 
hence 7 if understood as a projective subvariety of $\PP^{17}_\C$.
However, it is not clear whether $V(I)$ is smooth or at least irreducible --
{\sc Singular} crashed when calculating the 10-minors, and it timed-out
when trying the primary decomposition.

\subsubsection{Higgs facets}
\label{higgsFacets}
On the other hand, it is easily possible to calculate the commutator 
property for the three facets of the 
Higgs range polytope $\HR=\HR(\CT_{\PP^2})$. 
This leads to all Higgs fields $\Phi$ with maximal one-dimensional
Higgs polytopes $\nabla(\Psi)\subseteq\HR$.
According to the figure at the end of Subsection (\ref{endoKnot}),
we have named the 18 coordinates in the following way:
$$
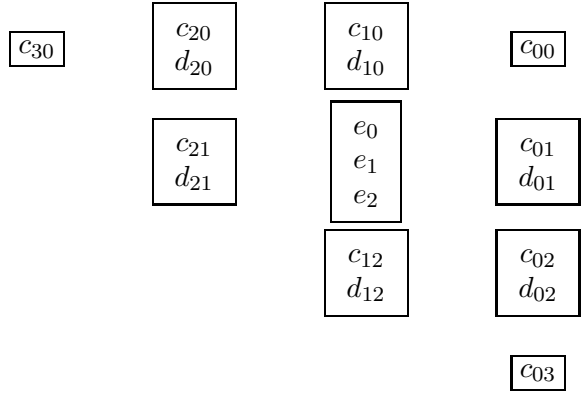
\begin{array}{c@{\hspace{3em}}c@{\hspace{3em}}c@{\hspace{3em}}c}
\fbox{$c_{30}$} & 
\fbox{$\begin{array}{c}c_{20}\\d_{20}\end{array}$} & 
\fbox{$\begin{array}{c}c_{10}\\d_{10}\end{array}$} &
\fbox{$c_{00}$} \\[3ex]
& \fbox{$\begin{array}{c}c_{21}\\d_{21}\end{array}$} &
\fbox{$\begin{array}{c}e_0\\e_1\\e_2\end{array}$} &
\fbox{$\begin{array}{c}c_{01}\\d_{01}\end{array}$} 
\\[4ex]
&& \fbox{$\begin{array}{c}c_{12}\\d_{12}\end{array}$} &
\fbox{$\begin{array}{c}c_{02}\\d_{02}\end{array}$}   
\\[5ex]
&&&\fbox{$c_{03}$}
\end{array}
$$
Then, the three facet ideals are
$$
\begin{array}{rcl}
I_0 &=& (-c_{03}\,d_{21}+c_{12}\,d_{12},\; -c_{21}\,d_{21}+c_{30}\,d_{12},\; 
-c_{21}\,c_{12}+c_{30}\,c_{03})
\\[1ex]
I_1 &=& (-c_{03}\,d_{01}+c_{02}\,d_{02},\; -c_{01}\,d_{01}+c_{00}\,d_{02},\;
-c_{01}\,c_{02}+c_{00}\,c_{03})
\\[1ex]
I_2 &=& (-c_{30}\,d_{10}+c_{20}\,d_{20},\; -c_{10}\,d_{10}+c_{00}\,d_{20},\; 
-c_{10}\,c_{20}+c_{00}\,c_{30}).
\end{array}
$$
All of them define a specific toric variety which is described, in each case, 
by a 3-dimensional, triangular prism. 
This can be seen by rewriting the binomial equations over the respective
tori as
$$
\textstyle
\frac{c_{12}}{c_{03}}=\frac{d_{21}}{d_{12}} = \frac{c_{30}}{c_{21}},
\hspace{2em}
\frac{c_{01}}{c_{00}}=\frac{d_{02}}{d_{01}} = \frac{c_{03}}{c_{02}},
\hspace{2em}
\frac{c_{00}}{c_{10}}=\frac{d_{10}}{d_{20}} = \frac{c_{20}}{c_{30}}.
$$

\subsubsection{Involving the center}
\label{toddler}
Here we will do the opposite of Subsection~(\ref{higgsFacets})
-- we keep the central variables $e_0, e_1, e_2$ of degree $0$
and the three corner degrees, i.e., the variables $c_{00}, c_{30}, c_{03}$.
This allows to approach the question raised in
Subsection~(\ref{sec:HiggsPoly}): Is there always a true Higgs bundle
having $\HR$ as its associated Higgs polytope? If this were
true, then all the corner degrees have to be involved.

\medskip

At this point we will additionally assume that the intermediate degrees 
do not occur, which is a non-trivial restriction though.
Thus, we have only six variables left, and a {\sc Singular} calculation
yields that the resulting Higgs variety consists of three 
projectively one-dimensional components 
$$
V(e_0-e_1,\; e_2,\; c_{00},\; c_{30}), \hspace{1em}
V(e_0,\; e_1-e_2,\; c_{30},\; c_{03}), \hspace{1em}
V(e_1,\; e_0-e_2,\; c_{00},\; c_{03})
$$
inside $\PP^5$, and that there are three embedded components, too.
In any case, the associated Higgs polytopes are the line segments connecting
a vertex of $\HR$ with the central point $0$.

\section{\peeref{Trace-free} $\CT_X$-Higgs fields on smooth complete surfaces}
\label{sec:MInSur}
Next we sketch how the techniques for $\P^2$ generalise to any smooth, complete toric surface $X$ for the computation of the Higgs range $\HR(X)$ and the associated vector spaces $V_r^0(X)$. 

\subsection{Encoding endomorphisms}
\label{encodeEndo2}
The fan of the Hirzebruch surface $\mathbb{H}_a$, $a\geq2$, is induced by the primitive generators $\rho_0(a)=(-1,-a)$, $\rho_1=(1,0)$, $\rho_2=(0,1)$ and $\rho_3=(0,-1)$. In the same vein as in the previous section we consider the lattice $\Z^3=\Z\rho_0(a)\oplus\Z\rho_1\oplus\Z\rho_2$, $a\geq1$, and identify
$$
N\cong\Z^3/\Z(1,1,a)
\hspace{1em}\mbox{and}\hspace{1em}
M\cong(1,1,a)^\bot=\{r\in\Z^3\kst r_0+r_1+ar_2=0\}\subseteq\Z^3.
$$
While we do not make use of this, we note in passing that
the rays $\rho_0(a)$, $\rho_1$ and $\rho_2$ provide the fan of the
singular weighted projective plane $\PP(1,1,a)$. 
Proceeding as above yields the representation
$$
\textstyle
\widetilde\phi=
\left(\begin{array}{ccc}
0 & ax & -x \\ -ay & 0 & y \\ z & -z & 0
\end{array}\right)
=x A_0(a) + y A_1(a) + z A_2
\leadsto
\phi_0=
\left(\begin{array}{cc}
-ax & x+y \\ -a^2x-z & ax
\end{array}\right),
$$
where 
$$
A_0(a)=\left(\begin{array}{ccc}
 0 & a &-1\\
 0 & 0 & 0\\ 
 0 & 0 & 0
\end{array}\right)
\hspace{2em}
A_1(a)=\left(\begin{array}{ccc}
 0 & 0 & 0\\
-a & 0 & 1\\
 0 & 0 & 0
\end{array}\right)
\hspace{2em}
A_2=\left(\begin{array}{ccc}
 0 & 0 & 0\\
 0 & 0 & 0\\
 1 &-1 & 0
\end{array}\right).
$$
Under this representation, the determinant is given by $\det(\phi_0)=a^2xy+yz+zx$.

\subsection{From filtrations to facets}
\label{filtFac2}
To compute a basis for the vector spaces $V_r^0(X)$, we recall from Section~\ref{filtFac} that an endomorphism $\varphi\in\gEnd(E)$ belongs to the class (i)$_\rho$ if $\varphi(\rho)\in\spann(\rho)$, and to the class (ii)$_\rho$ if $\varphi(E)\subseteq\spann(\rho)\subseteq\ker(\varphi)$. For instance, for $a\geq1$ an endomorphism of class (i)$_{\rho(a)}$ is determined by the eigenvector equation $(xA_0(a)+yA_1(a)+zA_2)\rho_0(a)=\lambda\rho_0(a)$, or equivalently, by the matrix equation
$$
\left(\begin{array}{ccc}
0 & ax &-x\\
-ay & 0 &y\\
z &-z & 0
\end{array}\right)\cdot
\left(\begin{array}{c}0\\1\\a\end{array}\right)=\left(\begin{array}{r}0\\\lambda\\a\lambda\end{array}\right)
$$
which implies $z=-a^2y$. Any such endomorphism is thus of the form 
$I_{\rho_0(a)}(x,y)=x A_0(a)+y(A_1(a)-a^2A_2)$ 
for $x$, $y\in\C$. 
Similarly, any endomorphism of class (ii)$_{\rho_0(a)}$ is given by 
$I\!I_{\rho_0(a)}(x)=x A_0(a)$. 
Table~\ref{tab:HiggsA2} displays the endomorphisms $I_\rho$ and $I\!I_\rho$ 
of type (i)$_{\rho}$ and (ii)$_{\rho}$ for the primitive generators $\rho_0(a)$, $\rho_1$ and $\rho_2$ of $\mathbb{H}_2$. 

\renewcommand{\arraystretch}{1.5}
\begin{table}[hbt]
\begin{tabular}[h]{|c|c|c|}
\hline
$\rho\in\Sigma_a(1)$ & Basis of all $I_\rho$ & Basis of all $I\!I_\rho$ \\
\hline
$\rho_0(a)$ & $A_0(a)$,\; $A_1(a)-a^2A_2$ & $A_0(a)$ \\
\hline
$\rho_1$ & $A_0(a)-a^2 A_2$,\; $A_1(a)$ & $A_1(a)$ \\
\hline
$\rho_2$ & $A_0(a)-A_1(a)$,\; $A_2$ & $A_2(a)$ \\
\hline
\end{tabular}
\medskip
\caption{The endomorphisms $I_\rho$ and $I\!I_\rho$ for $a\geq1$.}
\label{tab:HiggsA2}
\end{table}

\subsection{The Hirzebruch surfaces $\mathbb{H}_a$}
\label{HirzSurf}
With Table~\ref{tab:HiggsA2} at hand we can now determine $\HR(\mathbb{H}_a)$ with associated vector spaces $V_r^0(X)$ exactly in the same way as for the projective space. 

\begin{example}
For $\HR(\mathbb{H}_2)$ the Higgs range is the convex hull of the 
points $r^0=(1,0)$, $r^1=(-1,0)$, $r^2=(1,-2)$ and $r^3=(3,-2)$ 
given by the green polytope in the figure below:  
\begin{figure}[ht]
\begin{tikzpicture}[scale=0.65]
\draw[color=oliwkowy!40] (-3.3,-3.3) grid (6.3,3.3);
\fill[pattern color=green!50, pattern=north east lines]
  (-1,0) -- (1,0) -- (3,-2) -- (1,-2) -- cycle;
\draw[thick,  color=red]
  (-2.2,-2) -- (6.2,-2) (-2.2,2) -- (6.2,2) (-2,-2.2) -- (-2,2.2) (-2.2,2.1) -- (6.2,-2.1);
\draw[thick,  color=blue]
  (-1.2,-1) -- (3.2,-1) (3.2,1) -- (-1.2,1) (-1,-1.2) -- (-1,1.2) (-1.2,1.1) -- (3.2,-1.1);
\draw[thick,  color=green] (0,0) circle (3pt);
\fill[thick,  color=black] 
  (1,0) circle (3pt) (1,-1) circle (3pt) (-1,0) circle (3pt) (0,-1) circle (3pt) (1,-2) circle (3pt) (2,-1) circle (3pt) (2,-2) circle (3pt) (3,-2) circle (3pt);
\draw[thick, color=black]
  (-1,0) -- (1,0) -- (3,-2) -- (1,-2) -- cycle;
\draw[thick,color=black] 
   (1.5,0.5) node{$r^0$} (-1.5,0.5) node{$r^1$}
   (0.5,-2.5) node{$r^2$} (3.5,-2.5) node{$r^3$};
\end{tikzpicture}
\caption{The Higgs range $\HR(\mathbb{H}_2)$}
\end{figure}
We display a basis for the vector spaces $V_r^0(\mathbb{H}_2)$ as in Figure~\ref{fig:HiggsFieldsP2}. The origin is indicated by the green circle, and the polytope of the Higgs range is given by the gray lines:
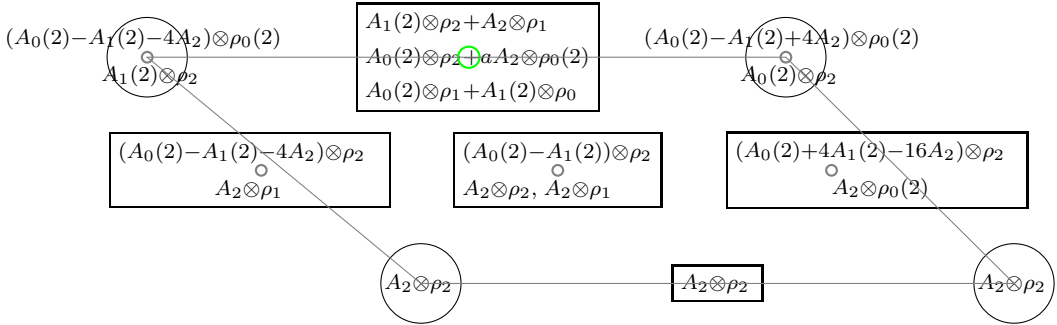
\begin{figure}[ht]
\newcommand{\coH}[2]{{\begin{minipage}{#1em}{$\ks #2$}\end{minipage}}}
\newcommand{\cid}[1]{{\color{red}s_#1\id_N}}
\newcommand{\aid}[1]{\mbox{ \tiny and }\cid{#1}}
\begin{tikzpicture}[x=3cm, y=1.5cm]
\draw[thick, color=gray]
  (-2,0) circle (2pt)  (0.8,0) circle (2pt)
  (-1.5,-1) circle (2pt) (-0.2,-1) circle (2pt) (1,-1) circle (2pt);
\draw[thick,color=black] 
   (-0.55,0) node{\fbox{$\coH{7.7}{A_1(2)\otimes\rho_2+A_2\otimes\rho_1\\A_0(2)\otimes\rho_2+aA_2\otimes\rho_0(2)\\A_0(2)\otimes\rho_1+A_1(2)\otimes\rho_0}$}}
   (1.2,-1) node{\fbox{$\coH{9.7}{(A_0(2)+4A_1(2)-16A_2)\otimes\rho_2\\[1pt]\phantom{aaaaaaaa}A_2\otimes\rho_0(2)}$}}
   (-1.55,-1) node{\fbox{$\coH{9.0}{(A_0(2)-A_1(2)-4A_2)\otimes\rho_2\\[1pt]\phantom{aaaaaaaa}A_2\otimes\rho_1}$}}
   (0.5,-2.0) node{\fbox{$\coH{2.5}{A_2\otimes\rho_2}$}}
   (-0.2,-1) node{\fbox{$\coH{6.5}{(A_0(2)-A_1(2))\otimes\rho_2\\[1pt]
               A_2\otimes\rho_2,\; A_2\otimes\rho_1}$}};
\draw[thin,color=black] 
   (0.8,0) circle (15pt) node{$\coH{9.7}{(A_0(2)-A_1(2)+4A_2)\otimes\rho_0(2)\\[1pt]\phantom{aaaaaaaa}A_0(2)\otimes\rho_2}$} 
   (-2,0) circle (15pt) node{$\coH{9.7}{(A_0(2)-A_1(2)-4A_2)\otimes\rho_0(2)\\[1pt]\phantom{aaaaaaaa}A_1(2)\otimes\rho_2}$}
   (-0.8,-2.0) circle (15pt) node{$\coH{2.5}{A_2\otimes\rho_2}$}
   (1.8,-2.0) circle (15pt) node{$\coH{2.5}{A_2\otimes\rho_2}$};
\draw[thick,  color=green] (-0.59,0) circle (4pt);
\draw[thin,  color=gray]
  (-2,0) -- (0.8,0) -- (1.8,-2.0) -- (-0.8,-2.0) -- cycle;
\end{tikzpicture}
\caption{\peeref{The trace-free Higgs fields on $\mathbb{H}_2$. For each degree the box respectively the circle at $r\in M$ contains a basis for the linear pre-Higgs fields of degree $r$, cf.\ also Table~\ref{tab:HiggsA2} and the caption from Figure~\ref{fig:HiggsFieldsP2}}.}
\end{figure}
\end{example}

For general $a$ we obtain the Higgs range of $\mathbb{H}_a$ as follows:
\begin{enumerate}
\item[(i)] We keep the lattice points $(-1,0)$, $(0,0)$, $(1,0)$, $(0,-1)$, $(1,-1)$ and $(1,-2)$ together with $V_r^0(\mathbb{H}_a)=V_r^0(\mathbb{H}_2)$. 
\item[(ii)] We add $a-1$ points $(2,-1),\ldots,(a,-1)$ 
with $V_{(x,-1)}^0(\mathbb{H}_a)=V_{(1,-1)}^0(\mathbb{H}_2)$ 
for $x\leq a-1$ and 
$V_{(a,-1)}^0=\spann((A_0(a)+a^2(A_1(a)-a^2A_2))\otimes\rho_2,\;
A_2\otimes\rho_0(a))$. 
  \item[(iii)] Finally, we add $2(a-1)$ points $(2,-2),\ldots,(2a-1,-2)$ with vector space $V_r^0(\mathbb{H}_a)=\spann(A_2\otimes\rho_2)$. 
\end{enumerate}

\begin{example}
For instance, we find the following Higgs range for $\mathbb{H}_4$ (where we only indicated the dimensions of $V_r^0$:
$$
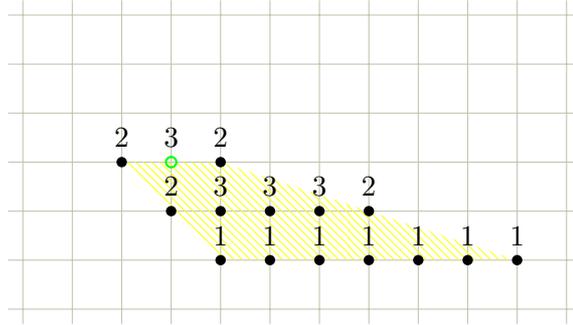
\begin{tikzpicture}[scale=0.65]
\draw[color=oliwkowy!40] (-3.3,-3.3) grid (8.3,3.3);
\fill[pattern color=yellow!90, pattern=north west lines]
  (-1,0) -- (1,0) -- (7,-2) -- (1,-2) -- cycle;
\fill[thick,  color=black] 
  (-1,0) circle (3pt) (1,0) circle (3pt) (0,-1) circle (3pt) (1,-1) circle (3pt) (2,-1) circle (3pt) (3,-1) circle (3pt) (4,-1) circle (3pt) (1,-2) circle (3pt) (2,-2) circle (3pt) (3,-2) circle (3pt) (4,-2) circle (3pt) (5,-2) circle (3pt) (6,-2) circle (3pt) (7,-2) circle (3pt);
\draw[thick,  color=green] (0,0) circle (3pt);
\draw[thick, color=black]  (1,0.5) node{$2$};
\draw[thick, color=black]  (0,0.5) node{$3$};
\draw[thick, color=black]  (-1,0.5) node{$2$};
\draw[thick, color=black]  (0,-0.5) node{$2$};
\draw[thick, color=black]  (4,-0.5) node{$2$};

\draw[thick, color=black]  (1,-0.5) node{$3$};
\draw[thick, color=black]  (2,-0.5) node{$3$};
\draw[thick, color=black]  (3,-0.5) node{$3$};

\draw[thick, color=black]  (1,-1.5) node{$1$};
\draw[thick, color=black]  (2,-1.5) node{$1$};
\draw[thick, color=black]  (3,-1.5) node{$1$};
\draw[thick, color=black]  (4,-1.5) node{$1$};
\draw[thick, color=black]  (5,-1.5) node{$1$};
\draw[thick, color=black]  (6,-1.5) node{$1$};
\draw[thick, color=black]  (7,-1.5) node{$1$};
\end{tikzpicture}
$$
 
\end{example}

\begin{remark}
Similarly, we can deal with $\PP^1\times\PP^1$ after we compute the corresponding basis for endomorphisms of type (i)$_\rho$ and (ii)$_\rho$.
\end{remark}

\subsection{The Higgs range under blow-ups}
\label{subsec:HRBlow}
A general smooth and complete toric surface $X$ is obtained from a minimal
surface by a finite sequence of blow-ups at fixed points of the torus action, so we need to discuss how such blow-ups affect $\HR(X)$.

\medskip

Combinatorically, the blow-up $X'\to X$ of a surface $X$ arises from
subdividing a maximal cone given by $\rho$, $\tau\in\Sigma(1)$ by inserting
the primitive generator $\sigma=\rho+\tau$. This yields new lines
$F^\sigma_i$, $i=0, 1, 2$ which possibly exclude further points of $\HR(X)$ 
or decrease the dimension of the corresponding vector spaces $V_r^0$ 
by adding further linear dependencies, cf.\ Subsection~(\ref{linDeps}). 
We therefore have the

\begin{proposition}
If $X'\to X$ is the blow-up of a smooth toric surface, then we have natural inclusions $\HR(X')\subset\HR(X)$ and $V_r^0(X')\subset V_r^0(X)$ whenever $r\in\HR(X')$. 
\end{proposition}

\subsection{Fano surfaces}
\label{subsec:Fano}
The precise shape of $\HR(X')$ depends of course on the combinatorics of 
$\HR(X)$ and on the fixed point we blow up. For illustration, consider 
the five smooth toric Fano surfaces given by the reflexive del Pezzo polytopes:
$$
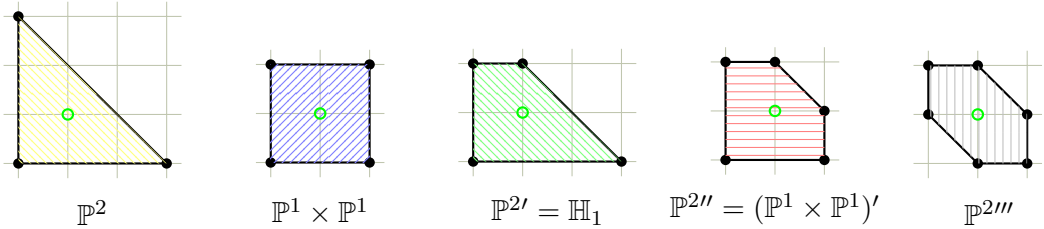
\newcommand{\abstZ}{\hspace{2.5em}}
\begin{tikzpicture}[scale=0.65]
\draw[color=oliwkowy!40] (-0.3,-0.3) grid (3.3,3.3);
\draw[thick,  color=black]
  (0,0) -- (3,0) -- (0,3) -- cycle;
\fill[thick,  color=black] 
  (0,0) circle (3pt) (3,0) circle (3pt) (0,3) circle (3pt);
\draw[thick,  color=black]  (1.5,-1.0) node{$\PP^2$};
\fill[pattern color=yellow!50, pattern=north west lines]
  (0,0) -- (3,0) -- (0,3) -- cycle;
\draw[thick,  color=green] (1,1) circle (3pt);
\end{tikzpicture}
\abstZ
\begin{tikzpicture}[scale=0.65]
\draw[color=oliwkowy!40] (-0.3,-0.3) grid (2.3,2.3);
\draw[thick,  color=black]
  (0,0) -- (2,0) -- (2,2) -- (0,2) -- cycle;
\fill[thick,  color=black] 
  (0,0) circle (3pt) (2,0) circle (3pt) (2,2) circle (3pt) (0,2) circle (3pt);
\draw[thick,  color=black]  (1.0,-1.0) node{$\PP^1\times\PP^1$};
\fill[pattern color=blue!50, pattern=north east lines]
  (0,0) -- (2,0) -- (2,2) -- (0,2) -- cycle;
\draw[thick,  color=green] (1,1) circle (3pt);
\end{tikzpicture}
\abstZ
\begin{tikzpicture}[scale=0.65]
\draw[color=oliwkowy!40] (-0.3,-0.3) grid (3.3,2.3);
\draw[thick,  color=black]
  (0,0) -- (3,0) -- (1,2) -- (0,2) -- cycle;
\fill[thick,  color=black] 
  (0,0) circle (3pt) (3,0) circle (3pt) (1,2) circle (3pt) (0,2) circle (3pt);
\draw[thick,  color=black]  (1.5,-1.0) node{$\PP^2{'}={\mathbb H}_1$};
\fill[pattern color=green!50, pattern=north west lines]
  (0,0) -- (3,0) -- (1,2) -- (0,2) -- cycle;
\draw[thick,  color=green] (1,1) circle (3pt);
\end{tikzpicture}
\abstZ
\begin{tikzpicture}[scale=0.65]
\draw[color=oliwkowy!40] (-0.3,-0.3) grid (2.3,2.3);
\draw[thick,  color=black]
  (0,0) -- (2,0) -- (2,1) -- (1,2) -- (0,2) -- cycle;
\fill[thick,  color=black] 
  (0,0) circle (3pt) (2,0) circle (3pt) (2,1) circle (3pt) 
  (1,2) circle (3pt) (0,2) circle (3pt);
\draw[thick,  color=black]  (1.0,-1.0)
                  node{\makebox[0em]{$\PP^2{''}=(\PP^1\times\PP^1){'}$}};
\fill[pattern color=red!50, pattern=horizontal lines]
  (0,0) -- (2,0) -- (2,1) -- (1,2) -- (0,2) -- cycle;
\draw[thick,  color=green] (1,1) circle (3pt);
\end{tikzpicture}
\abstZ
\begin{tikzpicture}[scale=0.65]
\draw[color=oliwkowy!40] (-0.3,-0.3) grid (2.3,2.3);
\draw[thick,  color=black]
  (0,1) -- (1,0) -- (2,0) -- (2,1) -- (1,2) -- (0,2) -- cycle;
\fill[thick,  color=black] 
  (0,1) circle (3pt) (2,0) circle (3pt) (2,1) circle (3pt) 
  (1,0) circle (3pt) (1,2) circle (3pt) (0,2) circle (3pt);
\draw[thick,  color=black]  (1.0,-1.0)
                  node{$\hspace*{0.8em}\PP^2{'''}$};
\fill[pattern color=gray!50, pattern=vertical lines]
  (0,1) -- (1,0) -- (2,0) -- (2,1) -- (1,2) -- (0,2) -- cycle;
\draw[thick,  color=green] (1,1) circle (3pt);
\end{tikzpicture}
$$
Their Higgs ranges together with the dimension of $V_r^0(X)$ are given as follows:
$$
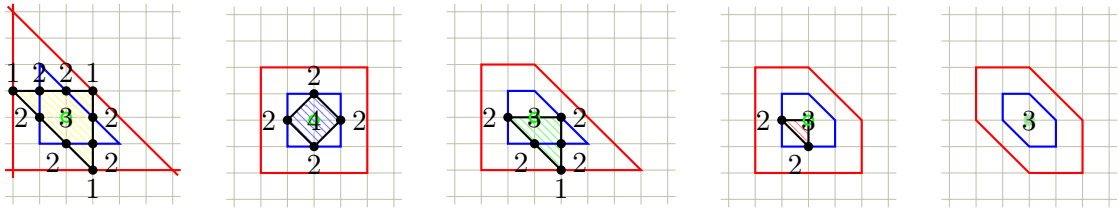
\newcommand{\abstZ}{\hspace{1.5em}}
\begin{tikzpicture}[scale=0.35]
\draw[color=oliwkowy!40] (-2.3,-3.3) grid (4.3,4.3);
\fill[pattern color=yellow!50, pattern=north west lines]
(-2,1) -- (1,1) -- (1,-2) -- cycle;
\draw[thick,  color=blue]
  (-1,-1) -- (2,-1) -- (-1,2) -- cycle;
\draw[thick,  color=red]
  (-2,4.3) -- (-2,-2.3) (-2.3,-2) -- (4.3,-2) (-2.2,4.2) -- (4.2,-2.2);
\draw[thick,  color=black]
(-2,1) -- (1,1) -- (1,-2) -- cycle;
\fill[thick,  color=black] 
  (1,0) circle (5pt) (1,-1) circle (5pt) (1,-2) circle (5pt) (0,-1) circle (5pt) (-1,0) circle (5pt) (-2,1) circle (5pt) (-1,1) circle (5pt) (0,1) circle (5pt) (1,1) circle (5pt);
\draw[thick,  color=green] (0,0) circle (5pt);
\draw[thick, color=black]  (1,1.7) node{$1$};
\draw[thick, color=black]  (1.7,0.0) node{$2$};
\draw[thick, color=black]  (1.7,-1.7) node{$2$};
\draw[thick, color=black]  (1,-2.7) node{$1$};
\draw[thick, color=black]  (0,0) node{$3$};
\draw[thick, color=black]  (-0.5,-1.7) node{$2$};
\draw[thick, color=black]  (0,1.7) node{$2$};
\draw[thick, color=black]  (-1.7,0) node{$2$};
\draw[thick, color=black]  (-2,1.7) node{$1$};
\draw[thick, color=black]  (-1,1.7) node{$2$};
\end{tikzpicture}
\abstZ
\begin{tikzpicture}[scale=0.35]
\draw[color=oliwkowy!40] (-3.3,-3.3) grid (3.3,4.3);
\fill[pattern color=blue!50, pattern=north west lines]
  (1,0) -- (0,1) -- (-1,0) -- (0,-1) -- cycle;
\draw[thick,  color=blue]
  (-1,-1) -- (1,-1) -- (1,1) -- (-1,1) -- cycle;
\draw[thick,  color=red]
  (-2,2) -- (-2,-2) -- (2,-2) -- (2,2) -- cycle;
\draw[thick,  color=black]
  (1,0) -- (0,1) -- (-1,0) -- (0,-1) -- cycle;
\fill[thick,  color=black] 
  (1,0) circle (5pt) (0,1) circle (5pt) (-1,0) circle (5pt) (0,-1) circle (5pt);
\draw[thick,  color=green] (0,0) circle (5pt);
\draw[thick, color=black]  (0,0) node{$4$};
\draw[thick, color=black]  (1.7,0) node{$2$};
\draw[thick, color=black]  (-1.7,0) node{$2$};
\draw[thick, color=black]  (0,-1.7) node{$2$};
\draw[thick, color=black]  (0,1.7) node{$2$};
\end{tikzpicture}
\abstZ
\begin{tikzpicture}[scale=0.35]
\draw[color=oliwkowy!40] (-3.3,-3.3) grid (5.3,4.3);
\fill[pattern color=green!50, pattern=north west lines]
  (1,0) -- (1,-2) -- (-1,0) -- cycle;
\draw[thick,  color=blue]
  (-1,-1) -- (2,-1) -- (0,1) -- (-1,1) -- cycle;
\draw[thick,  color=red]
  (-2,2) -- (-2,-2) -- (4,-2) -- (0,2) -- cycle;
\draw[thick,  color=black]
  (1,0) -- (1,-2) -- (-1,0) -- cycle;
\fill[thick,  color=black] 
  (1,0) circle (5pt) (1,-1) circle (5pt) (1,-2) circle (5pt) (0,-1) circle (5pt) (-1,0) circle (5pt);
\draw[thick,  color=green] (0,0) circle (5pt);
\draw[thick, color=black]  (1.7,0.0) node{$2$};
\draw[thick, color=black]  (1.7,-1.7) node{$2$};
\draw[thick, color=black]  (1,-2.7) node{$1$};
\draw[thick, color=black]  (0,0) node{$3$};
\draw[thick, color=black]  (-0.5,-1.7) node{$2$};
\draw[thick, color=black]  (-1.7,0) node{$2$};
\end{tikzpicture}
\abstZ
\begin{tikzpicture}[scale=0.35]
\draw[color=oliwkowy!40] (-3.3,-3.3) grid (3.3,4.3);
\fill[pattern color=red!50, pattern=north west lines]
  (0,0) -- (0,-1) -- (-1,0) -- cycle;
\draw[thick,  color=blue]
  (1,0) -- (1,-1) -- (-1,-1) -- (-1,1) -- (0,1) -- cycle;
\draw[thick,  color=red]
  (2,0) -- (2,-2) -- (-2,-2) -- (-2,2) -- (0,2) -- cycle;
\draw[thick,  color=black]
  (0,0) -- (0,-1) -- (-1,0) -- cycle;
\fill[thick,  color=black] 
  (0,-1) circle (5pt) (-1,0) circle (5pt);
\draw[thick,  color=green] (0,0) circle (5pt);
\draw[thick, color=black]  (0,0) node{$3$};
\draw[thick, color=black]  (-0.5,-1.7) node{$2$};
\draw[thick, color=black]  (-1.7,0) node{$2$};
\end{tikzpicture}
\abstZ
\begin{tikzpicture}[scale=0.35]
\draw[color=oliwkowy!40] (-3.3,-3.3) grid (3.3,4.3);
\draw[thick,  color=blue]
  (1,0) -- (1,-1) -- (0,-1) -- (-1,0) -- (-1,1) -- (0,1) -- cycle;
\draw[thick,  color=red]
  (2,0) -- (2,-2) -- (0,-2) -- (-2,0) -- (-2,2) -- (0,2) -- cycle;
\draw[thick,  color=green!50] (0,0) circle (5pt);
\draw[thick,  color=gray!50] (0,0) circle (3pt);
\draw[thick, color=black]  (0,0) node{$3$};
\end{tikzpicture}
$$
In particular, whenever \peeref{the vector spaces are non-trivial,}
we find $V_r^0(X')=V_r^0(\PP^2)$ for $X'=\PP^2{'}$, $\PP^2{''}$ and $\PP^2{'''}$.

\subsection{Higgs fields and their Higgs algebras on $\PP^2{''}$ and $\PP^2{'''}$}
\label{subsec:HiggsPolyDP}
Using Subsection~(\ref{noPre}) we exhibit some explicit Higgs fields on the del Pezzos $X=\PP^2{''}$ and $\PP^2{'''}$ and compute their associated invariants. 

\medskip 

We start with the degree six del Pezzo $\PP^2{'''}$ where $\HR(\PP^2{'''})=\{(0,0)\}$. 
From Subsection~(\ref{toddler}) we gather that the space of Higgs fields is given by the three components $V(e_0-e_1, e_2)$, $V(e_0, e_1-e_2, )$ and $V(e_1, e_0-e_2)$ with corresponding Higgs fields 
\begin{align*}
\Phi_1&=A_2\otimes\rho_0+A_2\otimes\rho_1+(A_0+A_1)\otimes\rho_2\\
\Phi_2&=(A_1+A_2)\otimes\rho_0+A_0\otimes\rho_1+A_0\otimes\rho_2\\
\Phi_3&=A_1\otimes\rho_0+(A_0+A_2)\otimes\rho_1+A_1\otimes\rho_2.
\end{align*}
The resulting Higgs polytopes $\nabla(\Phi_i)$, $i=1,\,2,\,3$, coincide trivially with $\HR(\PP^2{'''})=\{(0,0)\}$. Since $\det\Phi_i=-d^2$ for $i=1,\,2,\,3$, the Higgs algebras $\CA(\Phi_i)$ are all isomorphic to $\C[M][z]/(z^2-d^2)\cong\C[M]\times\C[M]$. It follows that $\CA\otimes_t\C\cong\C\oplus\C\cong\CA(t)$ is the product ring. The fibre of the spectral variety $\spec\CA(\Phi_i)\to T$ correspond thus of the two distinct eigenvalues of $\Phi_i$. 

\medskip

Next we turn to the degree seven del Pezzo $\PP^2{''}$. A {\sc Singular} 
aided computation yields for instance the three dimensional component $V(e_0-e_1,c_{12},d_{21},e_2)$ with corresponding Higgs field
$$
\Phi=\big(e_1(A_0-A_2)\chi^{(0,0)}+d_{12}(A_0-A_2)\chi^{(0,-1)}+c_{21}A_1\chi^{(-1,0)}\big)\otimes\rho_2.
$$
Depending on the concrete choice of the coefficients,
the Higgs polytope $\nabla(\Phi)$ realises every subpolytope of the 
Higgs range $\HR(\PP^2{''})=\{(-1,0),(0,0),(0,-1)\}$; 
generically, both polytopes coincide. 

\medskip

The Higgs algebras $\CA(\Phi)$ is again generated by a single 
Higgs field with minimal polynomial $\mu(z)=z^2+\det\Phi$. Since $\det\Phi=-d^2$ is a square in $\C[M]$, the resulting algebra $\CA(\Phi)$ is again isomorphic to $\C[M]\times\C[M]$. 
Furthermore, $\CA(\Phi)\otimes_t\C\cong\CA(t)$. The difference to the previous case is that up to isomorphism we have now two isomorphism types of $\CA(t)$: If $t\in T$ is a zero of $\delta$, then $\CA(t)=\C[\Phi(t)]\cong\C[x]/(x^2)$ for $\Phi$ is nilpotent. Otherwise, $\CA(t_0)$ is just the product ring $\C\times\C$.

\begin{remark}
We can also consider the determinant as a map from $\End(E)\otimes\C[M]\otimes N\to\C[M]\otimes S^\kbb N$ (the toric version of the {\em Hitchin map}~\cite{hi87},~\cite[Section 6]{si94II}). For our special examples the generators $\Phi_i$ have a triangular form which implies that the determinant $\det\Phi_i\in S^\kbb N\otimes\C[M]$ admits a square root in $N\otimes\C[M]$. \peeref{The latter defines two vector fields on $T$ whose images in the total space of the corresponding tangent bundle represent the spectral variety.}
\end{remark}

\bibliographystyle{alpha}
\bibliography{coHiggs}

\begin{thebibliography}{{Sim}94b}

\bibitem[BDPR20]{bdmr20}
Indranil {Biswas}, Arijit {Dey}, Mainak {Poddar}, and Steven {Rayan}.
\newblock {Toric co-Higgs bundles on toric varieties.}
\newblock {\em arXiv:2004.00730}, 2020.

\bibitem[DGPS19]{singular}
Wolfram Decker, Gert-Martin Greuel, Gerhard Pfister, and Hans Sch\"onemann.
\newblock {\sc Singular} {4-1-2} --- {A} computer algebra system for polynomial
  computations.
\newblock \url{http://www.singular.uni-kl.de}, 2019.

\bibitem[{Gua}11]{gu11}
Marco {Gualtieri}.
\newblock {Generalized complex geometry.}
\newblock {\em Ann. of Math. (2)}, 174 no. 1:75--123, 2011.

\bibitem[{Hit}87]{hi87}
Nigel {Hitchin}.
\newblock {The self-duality equations on a Riemann surface.}
\newblock {\em {Proc.\ London Math.\ Soc.}}, 55(3):59--126, 1987.

\bibitem[{Hit}11]{hi11}
Nigel {Hitchin}.
\newblock {Generalized holomorphic bundles and the B-field action.}
\newblock {\em Journal of Geometry and Physics}, 61:352--362, 2011.

\bibitem[{Kly}90]{klyachko}
Alexander {Klyachko}.
\newblock {Equivariant bundles on toral varieties.}
\newblock {\em {Math. USSR, Izv.}}, 35(2):337--375, 1990.

\bibitem[{Kly}02]{klyachkoICM}
Alexander {Klyachko}.
\newblock {Vector bundles, linear representations, and spectral problems.}
\newblock In {\em {Proceedings of the international congress of mathematicians,
  ICM 2002, Beijing, China, August 20--28, 2002. Vol. II: Invited lectures}},
  pages 599--613. Beijing: Higher Education Press; Singapore: World
  Scientific/distributor, 2002.

\bibitem[KM]{tropical}
Kiumars {Kaveh} and Christopher {Manon}.
\newblock {Toric bundles, valuations, and tropical geometry over semifield of
  piecewise linear functions}.
\newblock {arXiv:1907.00543 [math.AG]}.

\bibitem[{Pay}08]{payne}
Sam {Payne}.
\newblock {Moduli of toric vector bundles.}
\newblock {\em {Compos. Math.}}, 144(5):1199--1213, 2008.

\bibitem[{Ray}11]{ra11}
Steven {Rayan}.
\newblock Geometry of co-higgs bundles.
\newblock DPhil thesis, University of Oxford, 2011.

\bibitem[{Ray}14]{rayancoHiggs2}
Steven {Rayan}.
\newblock {Constructing co-Higgs bundles on $\PP^2$}.
\newblock {\em Q. J. Math.}, 65 no. 4:1437--1460, 2014.

\bibitem[RJS18]{parliaments}
Sandra~Di {Rocco}, Kelly {Jabbusch}, and Gregory {Smith}.
\newblock {Toric vector bundles and parliaments of polytopes}.
\newblock {\em Trans. Amer. Math. Soc.}, 370 no. 11:7715--7741, 2018.

\bibitem[{Sim}94a]{si94I}
Carlos {Simpson}.
\newblock {Moduli of representations of the fundamental group of a smooth
  projective variety I.}
\newblock {\em Publications math\'ematiques de l'I.H.E.S.}, tome 79:47--129,
  1994.

\bibitem[{Sim}94b]{si94II}
Carlos {Simpson}.
\newblock {Moduli of representations of the fundamental group of a smooth
  projective variety II.}
\newblock {\em Publications math\'ematiques de l'I.H.E.S.}, tome 80:5--79,
  1994.

\end{thebibliography}
\end{document}